\documentclass[12pt,a4paper]{article}

\usepackage{a4wide}

\usepackage[latin1]{inputenc}
\usepackage[T1]{fontenc}
\usepackage[english]{babel}

\usepackage{amsmath}
\usepackage{amssymb}
\usepackage{amsfonts,makeidx,amsthm,mathrsfs}

\newcommand{\Lr}{\mathscr{L}}
\newcommand{\dvol}{\frac{\omega^m}{m!}}

\newcommand{\dsubvol}{\frac{\omega^{m-1}}{(m-1)!}}
\newcommand{\dsubsubvol}{\frac{\omega^{m-2}}{(m-2)!}}

\newcommand{\tr}{\mathrm{tr}}
\newcommand{\omegalin}{\omega_{\textrm{lin}}}

\newcommand{\Walg}{\mathbb{W}}
\newcommand{\Sp}{\mathrm{Sp}}

\newcommand{\h}{\mathfrak{h}}

\newcommand{\R}{\mathbb{R}}

\newcommand{\N}{\mathbb{N}}
\newcommand{\Z}{\mathbb{Z}}

\def\dim{\mathop{\mathrm{dim}}\nolimits}
\def\ham{\mathop{\mathrm{Ham}}\nolimits}

\newcommand{\E}{\mathcal{E}}

\newcommand{\LC}{\mathrm{lv}}
\newcommand{\MOm}{\mathcal{M}_{\Theta}}
\newcommand{\grad}{\textrm{grad}}
\newcommand{\D}{\mathcal{D}}

\newcommand{\extwedge}{\stackrel{\circ}{\wedge}}

\newcommand{\W}{\mathcal{W}}

\newtheorem{theorem}{Theorem}[section]

\newtheorem{prop}[theorem]{Proposition}
\newtheorem{defi}[theorem]{Definition}

\theoremstyle{definition} 
\newtheorem{ex}[theorem]{Example}
\theoremstyle{remark}

\newtheorem{rem}[theorem]{Remark}

\usepackage{color}

\newcommand{\mapright}[1]{\smash{\mathop{   \hbox to 0.7cm{\rightarrowfill}}
  \limits^{#1}}}
\newcommand{\mapdown}[1]{\Big\downarrow\rlap{$\vcenter{\hbox{$\scriptstyle#1\,$}}$ }}

\newcommand{\barpartial}{{\overline \partial}}
\newcommand{\bfC}{{\mathbb C}}
\newcommand{\bfP}{{\mathbb P}}
\newcommand{\bfR}{{\mathbb R}}
\newcommand{\bfZ}{{\mathbb Z}}

\newcommand{\bari}{{\overline i}}
\newcommand{\barj}{{\overline j}}
\newcommand{\bark}{{\overline k}}

\newcommand{\Ric}{\mathrm {Ric}}

\begin{document}

\renewcommand{\refname}{Bibliography}

\title{Deformation quantization and K\"ahler geometry with moment map}

\author{
Akito Futaki and Laurent La Fuente-Gravy\footnote{Partially supported by the IRP GEOMQ15 of the University of Luxembourg and the OPEN project QUANTMOD of the FNR}\\
{\small Yau Mathematical Sciences Center, Tsinghua University, Haidian district, Beijing 100084, China}\\
{\small futaki@tsinghua.edu.cn}
\and
 {\small Mathematics Research Unit, Universit\'e du Luxembourg,}\\
{\small MNO, 6, Avenue de la Fonte, L-4364 Esch-sur-Alzette, Luxembourg,\ \ \ }\\
{\small laulafuent@gmail.com}
}

\maketitle

\begin{abstract}
In the first part of this paper we outline the constructions and properties of Fedosov star product and
Berezin-Toeplitz star product. In the second part  we outline 
the basic ideas and recent developments on Yau-Tian-Donaldson conjecture
on the existence of K\"ahler metrics of constant scalar curvature. In the third part of the paper we outline recent results of both authors, and in particular show that the constant scalar curvature K\"ahler metric problem and the study of deformation quantization meet at the notion of trace (density) for star product. We formulate a cohomology formula for  the invariant of K-stability condition on K\"ahler metrics with constant Cahen-Gutt momentum.
\end{abstract}

\noindent {\footnotesize {\bf Keywords:} Symplectic connections, Moment map, Deformation quantization, Closed star products, K\"ahler manifolds.}

\newpage

\tableofcontents


\section{Introduction}


Let $M$ be a closed K\"ahler manifold. The existence of constant scalar 
curvature K\"ahler metric has been extensively studied over many years, and is still
one of the main problems in complex geometry.

The scalar curvature of a compact K\"ahler manifold gives a moment map in the Donaldson-Fujiki picture. Since the zeros of the moment map are considered to correspond to the stable orbit in Geometric Invariant Theory by Kempf-Ness theorem \cite{KempfNess} it is conjectured that the existence of constant scalar curvature metrics should be equivalent to certain notion of GIT stability (the Yau-Tian-Donaldson conjecture).

In the balanced metric approach, Luo \cite{Luo} expressed a stability condition for a polarized K\"ahler manifold $(M,L)$ as the constancy of the Bergman function, see also \cite{phongsturm03}. Considering tensor powers $L^k$ of $L$ and corresponding Bergman functions $\rho_k$, Donaldson \cite{donaldson01} used the asymptotic expansion of the Bergman function to prove the following. Assume that $\mathrm{Aut}(M,L)$ is discrete, if there exists a constant scalar curvature K\"ahler form in $c_1(L)$ then
\begin{enumerate}
\item $(M,L)$ is asymptotically stable and thus for each $k$ a balanced metric of $L^k$ exists,
\item  as $k \to \infty$ the balanced metrics converge to the constant scalar curvature K\"ahler metric. 
\end{enumerate}

In the framework of prequantizable K\"ahler manifolds, it was shown by Bordemann-Meinrenken-Schlichenmaier \cite{BMS} that quantization by Toeplitz operators has the correct semi-classical behaviour. The asymptotic expansion of the composition of Toeplitz operators yields an associative formal deformation of the Poisson algebra of smooth functions of the symplectic manifold: the Berezin-Toeplitz star product, see \cite{Schlich} for a proof (this results was already obtained by Bordemann-Meinrenken-Schlichenmaier shortly after \cite{BMS}, confer the references in \cite{Schlich}). Before that, the Bergman function $\rho_k$ was already studied by Rawnsley \cite{Rawnsley}, when it is constant for all $k$, Cahen-Gutt-Rawnsley \cite{CGR} already obtained a deformation quantization of the K\"ahler manifold.

In the formal deformation quantization of a symplectic manifold, or more generally of a Poisson manifold, defined in \cite{BFFLS} the quantization procedure is an associative deformation of the Poisson algebra of observables. That is a star product $*$ on the space $C^{\infty}(M)[[\nu]]$ given by a series of bidifferential operators deforming the pointwise product and satisfying $\frac{1}{\nu}(F*G-G*F)-\{F,G\}=O(\nu)$. Star products do exist on any symplectic manifold by Dewilde--Lecomte \cite{DWL}, Fedosov \cite{fed} and Omori--Maeda--Yoshioka \cite{OMY} and more generally on any Poisson manifolds by Kontsevitch \cite{Kon}. 

In this survey, we show that the constant scalar curvature K\"ahler metric problem and the study of deformation quantization meet at the notion of trace (density) for star product. Star products on symplectic manifolds admit an essentially unique trace \cite{fed}, \cite{NT}, \cite{gr}, that is a character on the Lie algebra $(C^{\infty}_c(M)[[\nu]],[\cdot,\cdot]_*)$ for $[\cdot,\cdot]_*$ denoting the $*$-commutator. Moreover, the trace of a star product can always be written as an $L^2$-pairing with an essentially unique formal function $\rho\in C^{\infty}(M)[\nu^{-1},\nu]]$ (where we allow a finite number of negative powers of $\nu$), called a trace density. 

Connes-Flato-Sternheimer \cite{CFS} define strongly closed star products for which the integration functional is a trace. Equivalently, it means that the trace density is a formal constant, i.e. $\rho \in \R[\nu^{-1},\nu]]$. If such a closed star product exists, it is possible to define its character \cite{CFS}, a cyclic cocycle in cyclic cohomology. The character of a closed star product on a symplectic manifold has been identified first for cotangent bundles in \cite{CFS} and after that for closed symplectic manifold by Halbout using the index theorem for Fedosov star products \cite{fed3}, \cite{NT}.

Back to the settings of a closed prequantizable K\"ahler manifold $(M,L)$ with $\omega\in c_1(L)$ and the tensor powers $L^k$ of $L$ we consider the underlying Berezin-Toeplitz star product \cite{Schlich}. The Bergman kernel $\rho_k$, more precisely, a formal version of its asymptotic expansion $\rho\in C^{\infty}(M)[\nu^{-1},\nu]]$ (setting $\nu=\frac{1}{k}$ in the expansion) gives a trace density for the Berezin-Toeplitz star product. Following the Tian-Yau-Zelditch (TYZ) expansion \cite{Tian}, \cite{Zel}, the first possibly non-constant term in $\rho$ is a multiple of the scalar curvature of the K\"ahler manifold. So that for a Berezin-Toeplitz star product the closedness condition means the coefficients of TYZ expansion are constants, hence scalar curvature must be constant.

The above suggests that trace densities and closedness of ``naturally'' defined star products could be studied from a K\"ahler geometry point of view. Our ``naturally'' defined star products will be star products obtained from Fedosov's method \cite{fed}. Fedosov star products exist on any symplectic manifold and only depends on the choice of a symplectic connection on the symplectic manifold (we set extra possible choices equal to $0$). His method provides an algorithm to obtain the bidifferential operators defining the star product. We will present his method in this survey and perform it up to order $3$ in $\nu$.

Consider now the trace density of such a Fedosov star product. The second author \cite{LLF} identifies the first possibly non-constant term to be the image of a moment map on the infinite dimensional space of symplectic connections previously discovered by Cahen-Gutt \cite{cagutt}. We call this image, which will be described 
in \eqref{CG momentum} as $\mu(\nabla)$, the Cahen-Gutt momentum. 
In view of the TYZ expansion, it suggests an analogy between the Cahen-Gutt momentum and the scalar curvature. Pushing the analogy further, the second author \cite{La Fuente-Gravy 2016_2} defined a Futaki invariant obstructing the constancy of the Cahen-Gutt momentum on K\"ahler manifolds, and hence obstructing the existence of closed Fedosov star product on K\"ahler manifolds. This invariant is inspired from the work of the first author \cite{futaki83.1} in which an obstruction to the existence of constant scalar curvature K\"ahler metrics is discovered. Very recently, together with Ono, the first author \cite{FO_CahenGutt} proved an analogue of the Calabi-Lichnerowicz-Matsushima theorem for a Cahen-Gutt version of extremal K\"ahler metrics. As a byproduct, the non-reductiveness of the reduced Lie algebra of holomorphic vector fields (assuming a non-negativity condition on the Ricci tensor) is an obstruction to the existence of closed Fedosov star product on K\"ahler manifolds.

In view of the analogies with the constant scalar curvature K\"ahler metric problem discovered in \cite{LLF,La Fuente-Gravy 2016_2} and \cite{FO_CahenGutt}, we may expect that Geometric Invariant Theory would play some role also in the study of deformation quantization. In the last part of the paper we formulate a version of K-stability for 
the existence of K\"ahler metrics of constant Cahen-Gutt momentum.

\subsection*{Acknowledgement}

Part of this work was written during the second author was visiting Tsinghua University, second author would like to thank Tsinghua University and the first author for this opportunity and for hospitality.

\section{Deformation Quantization}
\subsection{Definition and general properties} \label{subsect:defandgeneral}

On a symplectic manifold $(M,\omega)$, a star product as defined in \cite{BFFLS} is a formal associative deformation of the Poisson algebra of functions $(C^{\infty}(M),\cdot, \{\cdot,\cdot\})$. Recall that the symplectic form $\omega$ is a closed nondegenerate $2$-form. It induces the Poisson bracket $\{F,G\}:=-\omega(X_F,X_G)$ for $F,G\in C^{\infty}(M)$ and vector field $X_F$ uniquely determined by $\imath(X_F)\omega=dF$. 

A \emph{star product} is a product on the space $C^{\infty}(M)[[\nu]]$ of formal power series in $\nu$ with coefficient in $C^{\infty}(M)$ defined by :
$$F*G:=\sum_{r=0}^{+\infty} \nu^r C_r(F,G), \textrm{ for } F,G\in C^{\infty}(M)[[\nu]]$$
such that :
\begin{enumerate}
\item $*$ is associative,
\item the $C_r$'s are bidifferential $\nu$-linear operators,
\item $C_0(F,G)=FG$ and $C_1^-(F,G):=C_1(F,G)-C_1(G,F)=\{F,G\}$,
\item the constant function $1$ is a unit for $*$ (i.e. $F*1=F=1*F$).
\end{enumerate}

The existence of star product on symplectic manifolds was first obtained by Dewilde--Lecomte \cite{DWL}, and also by Fedosov \cite{fed} and Omori--Maeda--Yoshioka \cite{OMY}. Kontsevitch \cite{Kon} proved the existence of star products on any Poisson manifold.

\begin{ex} \label{examplemoyal}
Consider the vector space $\R^{2n}$ endowed with linear symplectic structure 
$$\omegalin:=\frac{1}{2}(\omegalin)_{ij}dx^i\wedge dx^j.$$ 
The Moyal star product of $F$ and $G\in C^{\infty}(\R^{2n})$ is defined by:
\begin{eqnarray}
(F*_{\textrm{Moyal}}G)(x) & := &\left.\left( \exp\left(\frac{\nu}{2}\Lambda^{ij} \partial_{y^i} \partial_{z^j}\right) F(y)G(z) \right)\right|_{y=z=x} \nonumber \\
 & = & \sum_{r=0}^{+\infty} \left(\frac{\nu}{2}\right)^r \frac{1}{r!}\Lambda^{i_1j_1}\ldots \Lambda^{i_rj_r} \frac{\partial^r F}{\partial x^{i_1} \ldots \partial x^{i_r}}(x)\frac{\partial^r G}{\partial x^{j_1} \ldots \partial x^{j_r}}(x) \label{eq:Moyal}, \nonumber
\end{eqnarray}
where $\Lambda^{ij}$ denotes the coefficients of the inverse matrix of $(\omegalin)_{ij}$.
\end{ex}

On symplectic manifolds, the classification of star products up to equivalence was obtained by Bertelson--Cahen--Gutt \cite{bertcagutt},  Deligne \cite{del}, Nest--Tsygan \cite{NT}.

Let $*$ and $*'$ be two star products on $C^{\infty}(M)[[\nu]]$, they are said to be \emph{equivalent} if there exists a formal power series of differential operators $T$ of the form:
$$T=Id+\sum_{r=1}^{+\infty} \nu^r T_r,$$
such that 
$$T(F)*'T(G)=T(F*G).$$
Remark that such a series $T$ is invertible as a formal power series, so that if $*$ and $T$ are given, the above equation determines a star product $*'$.

\begin{theorem}[Bertelson--Cahen--Gutt \cite{bertcagutt},  Deligne \cite{del}, Nest--Tsygan \cite{NT}]
On a symplectic manifold $(M,\omega)$ the equivalence classes of star products are in bijection with the space $H^2_{\textrm{dR}}(M)[[\nu]]$ of formal power series in $\nu$ with coefficients in the second de Rham cohomology group of $M$.
\end{theorem}

\noindent It means, in particular, that the Moyal star product is the local model of star product up to equivalence.

Consider a star product $*$ and $C^{\infty}_c(M)$ be the space of smooth functions with compact support. A \emph{trace} for $*$ is a $\R[[\nu]]$-linear map 
\begin{equation*}
\tr:C^{\infty}_c(M)[[\nu]] \rightarrow \R[\nu^{-1},\nu]]:F\mapsto \tr(F):=\nu^l \sum_{r=0}^{\infty}\nu^r \tau_r(F)
\end{equation*}
such that $\tr([F,G]_*)=0$, for $[F,G]_*:=F*G-G*F$ with $F,G \in C^{\infty}_c(M)[[\nu]]$ and $l\in \Z$.

Denoting by $C_k^-$ the anti-symmetrization of the bidifferential operators $C_k$ defining $*$, one obtains a family of equation for the $\tau_r$'s:  for all $k\geq 0$ and $ F,G \in C^{\infty}_c(M)[[\nu]]$,
\begin{equation} \label{eq:eqtrace}
\tau_k(\{F,G\})+\tau_{k-1}(C_2^-(F,G))+\ldots + \tau_0(C_{k+1}^-(F,G))=0.
\end{equation}

\begin{theorem}[Fedosov \cite{fed}, Nest--Tsygan \cite{NT}, Gutt--Rawnsley \cite{gr}]
Any star product on a symplectic manifold $(M,\omega)$ admits a trace which is unique up to multiplication by an element of $\R[\nu^{-1},\nu]]$. Moreover, any traces is given by an $L^2$-pairing with a formal function $\rho\in C^{\infty}(M)[\nu^{-1},\nu]]$:
\begin{equation} \label{eq:defdens}
\tr(F)=\frac{1}{\nu^m}\int_M F\rho \dvol.
\end{equation}
\end{theorem}
The formal function $\rho$ in Equation (\ref{eq:defdens}) is called a \emph{trace density}. It is unique up to multiplication by an element of $\R[\nu^{-1},\nu]]$.

\begin{proof}[Sketch of proof]
We summarize the proof given in \cite{gr}. 

For uniqueness, observe that for $k=0$, Equation \ref{eq:eqtrace} becomes $\tau_0(\{F,G\})=0$ for all pairs of functions. Then, it is shown in \cite{BRW} that $\tau_0$ is the integration functional up to a multiple. Uniqueness follows from an induction. For any two traces of the form $\tau:=\sum_{r\geq 0} \nu^r\tau_r$ and $\tau':=\sum_{r\geq 0}\nu^r\tau'_r$ that coincide up to order $k-1\geq 1$, then there difference is also a trace, it means $\tau_k-\tau'_k$ vanishes on Poisson brackets and is then a multiple of the integral. Hence, there exists a constant $C$ such that $\tau$ and $(1+C\nu^k)\tau'$ coincide up to order $k$.

One way to construct traces for general star product is to patch together canonical traces for the local model of star product: the Moyal star product. 
Indeed, one observes that for $*_{\textrm{Moyal}}$ the Moyal star product on $(\R^{2m},\omega_0:=\sum_{i=1}^{m} dx^i\wedge dy^i)$ with standard coordinates $(x^i,y^i)$, the integral is a trace
$$\tr^{*_{\textrm{Moyal}}}(F):=\frac{1}{\nu^m}\int_M F \frac{\omega_0^m}{m!},\ \forall F\in C^{\infty}_c(\R^{2m})[[\nu]].$$
The factor $\frac{1}{\nu^m}$ normalises the trace functional in the following meaning. If $\xi$ is a conformal symplectic vector field on $(\R^{2m},\omega_0:=\sum_{i=1}^{n} dx^i\wedge dy^i)$, i.e. $\Lr_{\xi}\omega_0=\omega_0$, then the operator $D^{\xi}:=\Lr_{\xi}+\nu\frac{\partial}{\partial\nu}$ is a derivation of $*_{\textrm{Moyal}}$, i.e. $D^{\xi}(F*_{\textrm{Moyal}}G)=D^{\xi}(F)*_{\textrm{Moyal}}G+F*_{\textrm{Moyal}}D^{\xi}G$. The trace $\tr^{*_{\textrm{Moyal}}}$ is normalised in the sense that it satisfies the equation
\begin{equation*} \label{eq:normaltrace}
\tr^{*_{\textrm{Moyal}}}(D^{\xi}F)=\nu \frac{\partial}{\partial\nu}\tr^{*_{\textrm{Moyal}}}(F).
\end{equation*}

Now, if $\tilde{*}$ is any star product on $(\R^{2m},\omega_0)$, it is equivalent to $*_{\textrm{Moyal}}$ through an operator $T=Id+\nu\ldots$, such that $TF*_{\textrm{Moyal}}TG=T(F\tilde{*}G)$. One obtain a trace for $\tilde{*}$
\begin{equation*}
\tr^{\tilde{*}}(F):=\tr^{*_{\textrm{Moyal}}}(T(F))=\frac{1}{\nu^m}\int_M T(F) \frac{\omega_0^m}{m!},\ \forall F\in C^{\infty}_c(\R^{2n})[[\nu]].
\end{equation*}
The trace density of $\tr^{\tilde{*}}$ is $\rho^{\tilde{*}}:=T'(1)$ for $T'$ the formal adjoint of $T$ (with respect to $\frac{\omega_0^m}{m!}$), as by definition
$$\frac{1}{\nu^m}\int_M T(F)\frac{\omega_0^m}{m!}=\frac{1}{\nu^m}\int_M FT'(1) \frac{\omega_0^m}{m!}.$$
Note that the operator $D:=T^{-1}\circ D^{\xi}\circ T=\nu\frac{\partial}{\partial\nu}+\Lr_{\xi}+D'$, for $D'$ a formal differential operator, is a derivation of $\tilde{*}$. The trace $\tr^{\tilde{*}}$ satisfies the equation
\begin{equation} \label{eq:normaltrace*tilde}
\tr^{\tilde{*}}(DF)=\nu \frac{\partial}{\partial\nu}\tr^{\tilde{*}}(F).
\end{equation}

So that for a general symplectic manifold equipped with a star product $*$, one first constructs local traces on a Darboux open cover. Then, to globalise the local trace constructed one has to be sure they co\"incide on intersections of Darboux charts on a general symplectic manifold. For this, a normalisation condition introduced by Karabegov comes into play. Define \emph{$\nu$-Euler derivation} of the star product to be local derivation of the star product $*$ of the form
$$D:=\nu \frac{\partial}{\partial \nu} + \Lr_{\xi} + \hat{D},$$
for $\hat{D}$ a series of local differential operators and $\xi$ is a conformal symplectic vector field. A trace $\tr$ for a star product is called \emph{normalised} if it satisfies
$$\tr(DF)=\nu \frac{\partial}{\partial \nu} \tr(F),$$
in any open set $U$ and for any $\nu$-Euler derivation $D$ on $U$ (in fact, one is enough in any $U$). One then shows that normalised traces are unique. Finally, the local trace $\tr^{\tilde{*}}$ is normalised by Equation (\ref{eq:normaltrace*tilde}).
\end{proof}

A star product is called \emph{closed up to order $l$} if the integration map is a trace modulo terms in $\nu^{l+1}$:
\begin{equation*}
\int_M F*G\dvol=\int_M G*F\dvol + O(\nu^{l+1})
\end{equation*}
for all $F,G \in C^{\infty}_c(M)[[\nu]]$. We will say a star product is \emph{closed} if it is closed up to any order. It means that formal constants are the trace densities.

\begin{rem}
Our definition of closed star product corresponds to strongly closed star product in \cite{CFS}. We drop the strongly as our work is mainly focused on closedness up to order $3$ of Fedosov star product.
\end{rem}

Any star product on a symplectic manifold is equivalent to a strongly closed star product \cite{OMY2}. On a closed symplectic manifold $(M,\omega)$, Karabegov \cite{Kara} gave a short proof of how to get a closed star product equivalent to a given one. Consider a star product $*$ with trace density $\rho$ such that 
$$\int_M \rho \dvol = \int_M \dvol.$$
This is always possible as we can modify $\rho$ with a formal constant. Then, one obtains a formal exact $2m$-form $\rho\dvol- \dvol=\sum_{r\geq1}\nu^r\mu_r$, so that for all $r\geq 1$, $\mu_r=d\alpha_r$ for some $(2m-1)$-form $\alpha_r$. Set $X_r$ the vector fields such that $-i(X_r)\dvol=\alpha_r$ and define $B:= Id+\sum_{r\geq1} \nu^rX_r$ a formal series of first order differential operators. We can define a star product $\tilde{*}$ equivalent to $*$ by the formula $F\tilde{*}G:=B^{-1}(BF*BG)$ for $F,G \in C^{\infty}(M)[[\nu]]$. It is clear that $\tr^*$ determines a trace for $*'$ by the formula
$$\tr^{\tilde{*}}(BF)=\tr^*(F)=\int_M F \rho \dvol.$$
On the other hand, 
$$\int_M BF \dvol=\int_M F \dvol - \sum_{r\geq1} \nu^r\int_M F\Lr_{X_r}\dvol =\int_M F \rho \dvol,$$
where we used integration by part and then $\Lr _{X_r}\dvol=di(X_r)\dvol=-\mu_r$. It means that $\tr^{\tilde{*}}$ is the integral and hence $\tilde{*}$ is strongly closed.

\subsection{Moment maps and trace densities} \label{sect:momenttrace}

From above, we know trace densities lowest order term in $\nu$ is a constant. In this section, we focus on the next order term in $\nu$. It turns out that this term admits a moment map interpretation in various examples of star products \cite{LLF}.

\subsubsection*{The space of symplectic connections}

Symplectic connections are the main tool to construct star products on symplectic manifolds. Already in \cite{BFFLS} a symplectic connection is used to build a truncated star product up to order $3$ in $\nu$ (precisely the order we are interested in this subsection). Also, in \cite{gr3}, it is proved that star products (in fact the $C_2$ term) determines a unique symplectic connection.

A symplectic connection is a connection $\nabla$ on a symplectic manifold $(M,\omega)$ satisfying $\nabla \omega=0$ and $T^{\nabla}=0$, for $T^{\nabla}$ being the torsion tensor. There always exists a symplectic connection on a symplectic manifold. Indeed, consider a torsion free connection $\nabla^0$ on $M$ and define the tensor $N$ on $M$ by
$$\nabla^0_X\omega (Y,Z):=\omega(N(X,Y)Z).$$
Then, the connection $\nabla$ defined by 
$$\nabla_X Y:=\nabla^0_X Y+ \frac{1}{3}\left(N(X,Y)+N(Y,X)\right)$$
is a symplectic connection.

Moreover, any two symplectic connections $\nabla$ and $\nabla'$ will differ by $A(X):=\nabla_X-\nabla'_X$, for $A(\cdot)$ be a $1$-form with values in the endomorphism bundle of $TM$ such that 
\begin{equation*} \label{eq:Aunderline}
\underline{A}(X,Y,Z):=\omega(A(X)Y,Z) 
\end{equation*}
is completely symmetric, i.e. $\underline{A}\in \Gamma(S^3T^*M)$. Conversely, from any symplectic connection $\nabla$ and any $\underline{A}\in \Gamma(S^3T^*M)$ the connection $\nabla+A$ is symplectic. So that, the space $\E(M,\omega)$ of symplectic connections is the affine space
$$\E(M,\omega)=\nabla + \Gamma(S^3T^*M) \textrm{ for any symplectic connection }\nabla $$

From now on, we assume $M$ is closed. The space $\E(M,\omega)$ is naturally a symplectic space admitting a symplectic action of the group of Hamiltonian automorphisms. The symplectic form $\Omega^{\E}$ is the natural pairing of the symmetric $3$-tensors $\underline{A}$ and $\underline{B}$:
\begin{equation*}
\Omega^{\E}(A,B):=\int_M \Lambda^{i_1j_1} \Lambda^{i_2j_2} \Lambda^{i_3j_3}\underline{A}_{i_1i_2i_3}\underline{B}_{i_1i_2i_3} \dvol=\int_M \tr(A\extwedge B)\wedge \dsubvol
\end{equation*}
where $\extwedge$ is the wedge product on the form part and the composition of the endormorphism part. 

A symplectic diffeomorphism $\varphi$ acts on $\E(M,\omega)$  by:
$$(\varphi.\nabla)_X Y:=\varphi_*(\nabla_{\varphi^{-1}_* X}\varphi^{-1}_* Y),$$
for all $X,Y \in TM$ and $\nabla \in \E(M,\omega)$ and this action is symplectic. In particular, the group $\ham(M,\omega)$ acts symplectically on $\E(M,\omega)$.

Denote by $C^{\infty}_0(M)$ the space of smooth functions with zero integral. It is naturally identified to the Lie algebra of Hamiltonian vector fields through the relation $i(X_F)\omega=dF$ for $F \in C^{\infty}_0(M)$. The infinitesimal action of $-X_F$ on $\E(M,\omega)$ is simply the Lie derivative:
\begin{equation}\label{eq:Liedernabla}
\left(\Lr_{X_F}\nabla\right)(Y)Z=\nabla^2_{(Y,Z)}X_F + R^{\nabla}(X_F,Y)Z, 
\end{equation}
where $\nabla^2_{(U,V)}W:=\nabla_U\nabla_V W - \nabla_{\nabla_U V}W$ is the second covariant derivative and $R^{\nabla}(U,V)W:=[\nabla_U,\nabla_V]W-\nabla_{[U,V]}W$ is the curvature tensor of $\nabla$, for $U,V,W \in \Gamma(TM)$.

Recall the definition of the Ricci tensor $Ric^{\nabla}(X,Y):=\tr[V\mapsto R^{\nabla}(V,X)Y]$
for all $X,Y \in TM$. Set $P(\nabla)$ be the function defined, using a multiple of the first Pontryagin form of the manifold, by 
$$P(\nabla)\dvol:=\frac{1}{2}\tr(R^{\nabla}(.,.)\extwedge R^{\nabla}(.,.))\wedge \dsubsubvol.$$
Finally, define the  map $\mu:\mathcal{E}(M,\omega) \rightarrow C^{\infty}_0(M)$ by
\begin{equation}\label{CG momentum}
\mu(\nabla):=(\nabla^2_{(p,q)} Ric^{\nabla})^{pq} + P(\nabla)
\end{equation}
where the indices are raised using the symplectic form. We call $\mu(\nabla)$ the {\it Cahen-Gutt momentum} of 
$\nabla$. 

\begin{theorem} [Cahen--Gutt \cite{cagutt}] \label{theor:momentE}
The map $\mu:\mathcal{E}(M,\omega) \rightarrow C^{\infty}(M)$ is an equivariant moment map for the action of $\ham(M,\omega)$ on $\E(M,\omega)$, i.e.
\begin{equation} \label{eq:momentmu}
\left.\frac{d}{dt}\right|_{0} \int_M \mu(\nabla+tA)F\dvol=\Omega^{\E}_{\nabla}(\Lr_{X_F}\nabla,A).
\end{equation}
\end{theorem}

We now describe the link between the moment map $\mu$ and the trace density for star products. We will only consider truncated star products up to order $3$. In \cite{BFFLS}, it was shown that on any symplectic manifold, a generalisation of the Moyal $*$-product up to order $3$, which we will call truncated star product $*_{\nabla}^3$, can be obtained using a symplectic connection $\nabla$:
\begin{equation} \label{eq:fedtrunc}
F*_{\nabla}^3 G:= FG+\frac{\nu}{2}\{F,G\}+\frac{\nu^2}{8} \Lambda^{i_1j_1} \Lambda^{i_2j_2} \nabla^2_{i_1i_2} F\nabla^2_{j_1j_2} G+ \frac{\nu^3}{48}S^3_\nabla (F,G),
\end{equation}
for 
$$S^3_\nabla (F,G):= \Lambda^{i_1j_1} \Lambda^{i_2j_2} \Lambda^{i_3j_3}\underline{\Lr_{X_F}\nabla}_{i_1i_2i_3}\underline{\Lr_{X_G}\nabla}_{j_1j_2j_3}.$$
The bidifferential operator $S^3_\nabla$ is a cocycle for the Chevalley cohomology of $(C^{\infty}(M),\{\cdot,\cdot\})$ with respect to the adjoint representation onto itself. It is never exact and its cohomology class is independent of $\nabla$, see \cite{BFFLS}.

One geometric way to prolong formula (\ref{eq:fedtrunc}) to get a star product is to go through the Fedosov construction \cite{fed}, see next section.

\begin{prop} \label{prop:trunctrace}
The functional $F\mapsto \tau(F):=\int_M F \dvol - \frac{\nu^2}{24}\int_M F \mu(\nabla)\dvol$ gives a ``truncated'' trace for the truncated star product $*_{\nabla}^3$.
\end{prop}

\begin{proof}
Because the terms of \eqref{eq:fedtrunc} at order $0$ and $2$ in $\nu$ are symmetric in $F$ and $G$. One has
\begin{equation*} \label{eq:trunccommut}
[F, G]_{*_{\nabla}^3}:= \nu\{F,G\}+ \frac{\nu^3}{24}S^3_\nabla (F,G)
\end{equation*}

Now, consider a truncated trace functional $\tau(F):=\int_M F\dvol+\nu \tau_1(F)+\nu^2\tau_2(F)$ for the truncated star product $*_{\nabla}^3$. Clearly, Equation (\ref{eq:eqtrace}) is satisfied for $k=0$. For $k=1$, Equation (\ref{eq:eqtrace}) becomes
$$\tau_1(\{F,G\})=0,$$
which implies that $\tau_1$ is a multiple of the integral, see \cite{BRW}. As traces are unique up to multiplication by a formal constant, we can choose $\tau_1=0$. Now comes an interesting equation for $k=2$, we have
\begin{equation}\label{eq:tracek=2}
 \frac{1}{24}\int_M S^3_\nabla (F,G) \dvol + \tau_2(\{F,G\})=0.
\end{equation}
But the integral of $S^3_{\nabla}$ is the symplectic form $\Omega^{\E}$ so that the equation reduces to
\begin{equation*}
 \frac{1}{24}\Omega^{\E}_{\nabla}(\Lr_{X_F}\nabla,\Lr_{X_G}\nabla) =- \tau_2(\{F,G\}).
\end{equation*}
Finally, using the moment map Equation (\ref{eq:momentmu}) and then the equivariance, we have
\begin{eqnarray} \label{eq:tau2}
 \frac{1}{24}\int_M \mu_*(\Lr_{X_G}\nabla)F\dvol & = & - \tau_2(\{F,G\}), \nonumber \\
 -\frac{1}{24}\int_M \mu(\nabla)\Lr_{X_G}F\dvol & = & - \tau_2(\{F,G\}).
\end{eqnarray}
As $-\Lr_{X_G}F=\{F,G\}$, we see that the functional $\tau_2(H):=- \int_M H \mu(\nabla) \dvol$ for $H\in C_c^{\infty}(M)$  is a solution to Equation (\ref{eq:tau2}) and then satisfies the trace functional equation for $k=2$, that is Equation (\ref{eq:tracek=2}).
\end{proof}

\subsection{Fedosov construction}

Fedosov builds in \cite{fed2} a star product on any symplectic manifold. His construction is obtained by identifying $C^{\infty}(M)[[\nu]]$ with the algebra of flat sections of the Weyl bundle $\W$ endowed with a flat connection.

Consider the vector space $(\R^{2m},\omegalin:=\frac{1}{2}(\omegalin)_{ij}dx^i\wedge dx^j)$ as in Example \ref{examplemoyal}.
Let $\{y^i \left| i=1,\ldots,2m\}\right.$ a basis of the dual space $(\R^{2m})^*$. The \emph{formal Weyl algebra} $(\Walg,\circ)$ is the algebra over $\R[[\nu]]$ of formal power series of the form
\begin{equation} \label{eq:Weylalgel}
a(y,\nu):=\sum_{2k+l\geq 0} \nu^k a_{k,i_1\ldots i_l}y^{i_1}\ldots y^{i_l},
\end{equation}
two of its elements being multiplied using the \emph{Moyal star product}
\begin{eqnarray}
(a\circ b)(y,\nu) & := &\left.\left( \exp\left(\frac{\nu}{2}\Lambda^{ij} \partial_{y^i} \partial_{z^j}\right)a(y,\nu)b(z,\nu)\right)\right|_{y=z} \nonumber \\
 & = & \sum_{r=0}^{+\infty} \frac{1}{r!}\left(\frac{\nu}{2}\right)^r \Lambda^{i_1j_1}\ldots \Lambda^{i_rj_r} \frac{\partial^r a}{\partial y^{i_1} \ldots \partial y^{i_r}}\frac{\partial^r b}{\partial y^{j_1} \ldots \partial y^{j_r}} \label{eq:MoyalWeyl}. \nonumber
\end{eqnarray}
We assign degree $1$ to the variables $y^i$ and degree $2$ for the variable $\nu$, so that terms in Equation (\ref{eq:Weylalgel}) are ordered by degree.

The formal Weyl algebra is naturally equipped with an action of the symplectic linear group $\Sp:=\Sp(\R^{2m},\omegalin)$. That is, for a matrix $(A^i_j) \in \Sp$ and $a(y,\nu)$ as in Equation (\ref{eq:Weylalgel}), define
\begin{equation*}
\rho(A)a(y,\nu):=a(y\circ A^{-1},\nu)=\sum_{2k+l\geq 0} \nu^k a_{k,i_1\ldots i_l}(A^{-1})^{i_1}_{j_1} \ldots (A^{-1})^{i_l}_{j_l} y^{j_1}\ldots y^{j_l}.
\end{equation*}
Because $A$ preserves the symplectic form $\omegalin$, the action preserves the Moyal product $\circ$.
Its differential gives a Lie algebra action of the Lie algebra $\mathfrak{sp}$ of the Lie group $\Sp$.
\begin{equation*}
d\rho(B)a(y,\nu):=\frac{1}{2\nu}[\omega_{ji}B^i_ly^jy^l,a(y,\nu)]_{\circ},
\end{equation*}
for $B\in \mathfrak{sp}$ (i.e. $\omega_{ji}B^i_l$ is symmetric in $j,l$), where $[\cdot,\cdot]_{\circ}$ denotes the $\circ$-commutator. This action enables one define the formal Weyl bundle on any symplectic manifold and lift symplectic connections to it.

We now consider a symplectic manifold $(M,\omega)$. A symplectic frame at $x\in M$ is the data of a basis $\{e_i\left| i=1,\ldots,2m\}\right.$ of $T_xM$ such that $\omega(e_i,e_j)=(\omegalin)_{ij}$. The union of all symplectic frames at any point of $M$ forms a $\Sp$-principal bundle $F(M,\omega)$ called the symplectic frame bundle.

The \emph{formal Weyl bundle} is the vector bundle of Weyl algebra associated to the frame bundle:
$$\W:=F(M,\omega)\times_{\Sp} \Walg.$$
The sections of the Weyl bundle writes locally as formal power series:
\begin{equation*}
a(x,y,\nu):=\sum_{2k+l\geq 0} \nu^k a_{k,i_1\ldots i_l}(x)y^{i_1}\ldots y^{i_l},
\end{equation*}
where $a_{k,i_1\ldots i_l}(x)$ are, in the indices $i_1,\ldots,i_l$, the components of a symmetric tensor on $M$ and we call $2k+l$ the $\W$-degree (inherited from $\Walg$) of $ \nu^k a_{k,i_1\ldots i_l}(x)y^{i_1}\ldots y^{i_l}$. The space of sections of $\W$, denoted by $\Gamma \W$, has a structure of an algebra
defined by the fiberwise product
\begin{equation*}
(a \circ b)(x,y,\nu):= \Big( \exp(\frac{\nu}{2}\Lambda^{ij} \partial_{y^i} \partial_{z^j})a(x,y,\nu)b(x,z,\nu)\Big)|_{y=z}
\end{equation*}

To describe connections on $\W$ and curvature forms, we will consider the bundle $\W \otimes \Lambda(M)$ of forms with values in the Weyl algebra. Sections in $\Gamma \W \otimes \Lambda(M)$ admit local expression:
\begin{equation*}
\sum_{2k+l\geq 0} \nu^k a_{k,i_1\ldots i_l,j_1\ldots j_p}(x)y^{i_1}\ldots y^{i_l}dx^{j_1}\wedge \ldots \wedge dx^{j_p}.
\end{equation*}
The $a_{k,i_1\ldots i_l,j_1\ldots j_p}(x)$ are, in the indices $i_1,\ldots,i_l,j_1,\ldots,j_p $, the components of a tensor on $M$, symmetric in the 
$i$'s and antisymmetric in the $j$'s. The $\circ$-product extends to the space $\Gamma \W\otimes \Lambda^*(M)$, for $a\otimes \alpha$ and $b\otimes \beta\in \Gamma \W\otimes \Lambda^*(M)$, we define
$(a\otimes \alpha) \circ (b\otimes \beta) := a\circ b \otimes \alpha\wedge \beta$. The $\W$-valued forms inherits
the structure of a graded Lie algebra from the graded commutator $[s,s']_{\circ}:=s\circ s'- (-1)^{q_1q_2}s'\circ s$, where
$s$ is a $q_1$-form and $s'$ a $q_2$-form (anti-symmetric degree).

Consider now a symplectic connection $\nabla$ on $(M,\omega)$. It lifts to a connection $1$-form on the frame bundle $F(M,\omega)$ which induces a covariant derivative $\partial$ of $\Gamma \W$. Writing $\Gamma^k_{ij}$ the Christoffel symbols of the symplectic connection $\nabla$, the induced $\partial$ acts on sections as
\begin{equation*}
\partial a := da + \frac{1}{\nu}[\overline{\Gamma},a]_{\circ} \in \Gamma \W\otimes \Lambda^1M.
\end{equation*}
where $\overline{\Gamma}:=\frac{1}{2}\omega_{lk}\Gamma^k_{ij}y^ly^jdx^i$ (note that $\omega_{lk}\Gamma^k_{ij}$ is symmetric $l,\ j$ because $\nabla$ preserves the symplectic form). One extends $\partial$ to a graded derivation on $\Gamma \W \otimes \Lambda M$ using the Leibniz rule :
\begin{equation*}
\partial (a\otimes \alpha) := (\partial a) \wedge \alpha + a\otimes d\alpha.
\end{equation*}

The curvature $\partial \circ \partial$ of $\partial$ is expressed in terms
of the curvature tensor $R$ of the symplectic connection $\nabla$. 
\begin{equation*}
\partial\circ \partial a := \frac{1}{\nu}[\overline{R},a]_{\circ},
\end{equation*}
where $\overline{R}:= \frac{1}{4} \omega_{ir}R^r_{jkl}y^iy^jdx^k\wedge dx^l$.

Up to now, the connection $\partial$ is very particular as it comes from a $1$-form with values in the Lie algebra $\mathfrak{sp}$ realised as order $2$ elements in $\Gamma \W$. To make this connection flat, we will incorporate more general endomorphisms of $\Gamma \W$.  Define
\begin{equation*} \label{eq:deltadef}
\delta(a) := dx_k\wedge \partial_{y_k} a=-\frac{1}{\nu}[\omega_{ij}y^i dx^j,a]_{\circ}.
\end{equation*}
One checks that $\delta^2=0$ and $\delta\partial + \partial \delta=0$, moreover $\delta$ is a graded derivation of the $\circ$-product. We consider connection on $\Gamma \W$ of the form
\begin{equation} \label{eq:defD}
\D a:=\partial a - \delta a + \frac{1}{\nu}[r,a]_{\circ},
\end{equation}
where $r$ is a $\W$-valued $1$-form. Its curvature is given by 
\begin{equation*}
\D^2 a = \frac{1}{\nu}\left[\overline{R} + \partial r - \delta r + \frac{1}{2\nu}[r,r]_{\circ},a\right]_{\circ}.
\end{equation*}
So, the flatness of $\D$ is now an equation on the unknown $\W$-valued $1$-form $r$:
\begin{equation}\label{eq:req}
\overline{R} + \partial r - \delta r + \frac{1}{\nu}r\circ r = \Omega,
\end{equation}
using $2r\circ r=[r,r]_{\circ}$, for $\Omega \in \Omega^2(M)[[\nu]]$ being any closed formal $2$-form (which is central).

The key to solve Equation (\ref{eq:req}) is a Hodge decomposition of $\Gamma \W \otimes \Lambda(M)$. Define 
$$\delta^{-1} a_{pq}:= \frac{1}{p+q}y^ki(\partial_{x^k})a_{pq} \textrm{ if } p+q>0 \textrm{ and } \delta^{-1}a_{00}=0,$$
where $a_{pq}$ is a $q$-forms with $p$ $y$'s and $p+q>0$.
We then have the Hodge decomposition of $\Gamma \W\otimes \Lambda M$ :
\begin{equation} \label{eq:hodge}
\delta \delta^{-1}a + \delta^{-1} \delta a=a-a_{00}.
\end{equation}

\begin{theorem} 
For any given closed central $2$-form $\Omega \in \nu \Omega^2(M)[[\nu]]$,
there exists a unique solution $r \in \Gamma \W \otimes \Lambda^1 M$ of:
\begin{equation*}\label{eq:rtheor}
\overline{R} + \partial r - \delta r + \frac{1}{\nu}r\circ r = \Omega,
\end{equation*}
 with degree at least $3$ and satisfying $\delta^{-1}r=0$.
\end{theorem}

\begin{proof}
A solution $r$ with degree at least $3$ would have $r_{00}=0$. As $\delta^{-1}r=0$, the Hodge decomposition (\ref{eq:hodge}) gives the following equation for a solution $r$ of Equation (\ref{eq:req}):
\begin{equation} \label{eq:r}
r=\delta^{-1} \delta r=\delta^{-1}\left( \overline{R} - \Omega\right) + \delta^{-1} \left(\partial r+  \frac{1}{\nu}r\circ r \right).
\end{equation}
Since $\delta^{-1}$ raises the $y$ degree by $1$, this equation can be solved recursively and the solution is unique. Indeed, denoting by $r^{(k)}$ the degree $k$ component of $r$, then $\delta^{-1} \partial r^{(k)}$ has degree $k+1$ and $\delta^{-1} \left( \frac{1}{\nu}r\circ r\right)$ has degree $2k-1>k$ (when $k\geq 2$). So, starting with $r^{(3)}:= \delta^{-1}( \overline{R} - \Omega)$, one obtains $r^{(k)}$ for $k>3$ by induction. Such an $r$ is unique.

One checks that $r$ is indeed a solution of Equation (\ref{eq:req}).
\end{proof}

\begin{rem}
It is possible to incorporate connections with torsion in the Fedosov construction, see \cite{karaschlich}. In that case, the solution $r$ we are looking for will have non-zero $r^{(2)}$ that depends on the torsion.
\end{rem}

The equation (\ref{eq:r}) enables to compute $r$ recursively with respect to the $\W$-degree. However, the computation of high degree terms becomes more and more complicate and there is no nice formula available for it. We write $b^{(k)}$ the $\W$-degree $k$ component of $b\in \Gamma \W \otimes \Omega(M)$.

\begin{prop}
The solution $r$ to the Equation (\ref{eq:req}) satisfies the recursive equations :
\begin{eqnarray*}
r^{(3)} & = &\delta^{-1}\left( \overline{R} - \Omega^{(2)}\right) \\
r^{(k+3)} & = & -\delta^{-1}\left(\Omega^{(k+2)}\right) +\delta^{-1}\left(\partial r^{(k+2)}+ \frac{1}{\nu}\sum_{l=1}^{k-1} r^{(l+2)}\circ r^{(k+2-l)}\right), k\geq 1.
\end{eqnarray*}
In particular, when $\Omega=0$, one has :
\begin{eqnarray}
r^{(3)} & = & \frac{1}{8}\omega_{kr}R^{r}_{lij}y^ky^ly^i\otimes dx^j, \label{eq:r3} \\
r^{(4)} & = & \frac{1}{40} \omega\left(\partial_k, \left( \left(\nabla_p R^{\nabla}\right)(\partial_i,\partial_j)\right)\partial_l\right)y^ky^ly^iy^p\otimes dx^j. \label{eq:r4}
\end{eqnarray}
\end{prop}

Consider a flat connection $\D$ of the form (\ref{eq:defD}) it is a graded derivation of $\Gamma\W \otimes \Lambda M$. Then, $\Gamma \W_{\D} := \{a\in \Gamma \W\ |\ \D a=0\}$ is an algebra for the $\circ$-product called \emph{the algebra of flat sections}. Define the symbol map $\sigma :a\in \Gamma \W_{\D} \mapsto a_{00}\in C^{\infty}(M)[[\nu]]$. 
\begin{theorem} [Fedosov \cite{fed2}] 
The symbol map $\sigma$ is a bijection on flat sections with inverse $Q:C^{\infty}(M)[[\nu]]\rightarrow \Gamma \W_{\D}$.\\
Any $a \in \Gamma \W_{\D} $ is the unique solution to the equation:
\begin{equation*}
a=a_{00}+ \delta^{-1}\left(\partial a + \frac{1}{\nu}[r,a]_{\circ}\right).
\end{equation*}
For any $F\in C^{\infty}(M)[[\nu]]$, denote by $b^{(k)}$ the $\W$-degree $k$ component of $b\in \Gamma \W\otimes \Lambda(M)$, then:
\begin{eqnarray}
Q(F)^{(0)} & = & F \nonumber \\
Q(F)^{(k+1)} & = & \delta^{-1} \left(\partial Q(F)^{(k)} + \frac{1}{\nu} \sum_{l=1}^{k-1} [r^{(l+2)}, Q(F)^{(k-l)}]_{\circ}\right).\label{eq:QFrecu}
\end{eqnarray}
The $\circ$-product on $\Gamma \W_D$ induces a star product $*_{\nabla,\Omega}$ on $(M,\omega)$ by
\begin{equation*}
F*_{\nabla,\Omega}G:=\sigma(QF \circ QG) \textrm{ for } F,G \in C^{\infty}(M)[[\nu]].
\end{equation*}
\end{theorem}

\begin{rem}
Note that the map $Q$ depends on $\nabla$ and the series $\Omega$, through $\partial$ and $r$ .
\end{rem}

We will now give the star product $*_{\nabla,0}$ up to order $3$ in $\nu$ to show it is indeed a prolongation of the truncated star product defined by Equation (\ref{eq:fedtrunc}). In the following proposition, we give the first order terms of $Q(F)$ up to order $3$. For terms of higher degree, we only give the part that contribute to the order $\leq3$ terms of $*_{\nabla,0}$.

\begin{prop}
For $\nabla$ a symplectic connection, and for the trivial choice of closed formal $2$-form $\Omega=0$, one has:
\begin{equation*}
Q(F)^{(0)}  = F, \ \ \ 
Q(F)^{(1)}  =  \partial_k F y^k,\ \ \ 
Q(F)^{(2)}  =  \frac{1}{2}\nabla^2_{(l,k)}F y^ky^l,
\end{equation*}
\begin{equation*}\label{eq:QF3}
Q(F)^{(3)}  =  \left(\frac{1}{6} \underline{\Lr_{X_F}\nabla}_{pkl} - \frac{1}{8}\omega\left(R^{\nabla}(X_F,\partial_p)\partial_k,\partial_l\right)\right)y^ky^py^l. 
\end{equation*}
Writing $\simeq$ for equality modulo terms that will not contribute to the order $\leq3$ terms of $*_{\nabla,0}$, we get:
\begin{eqnarray}
Q(F)^{(4)} & \simeq & 0 \nonumber \\
Q(F)^{(5)} & \simeq & \frac{\nu^2}{3!.48}\Lambda^{kt}\Lambda^{lu}\Lambda^{iv}\left(\stackrel{\curvearrowright}{\oplus}_{kli}\omega_{kr}(R^{\nabla})^r_{lij}  \right)\left(\stackrel{\curvearrowright}{\oplus}_{tuv} \underline{\Lr_{X_F}\nabla}_{tuv}\right)y^j \nonumber\\
 & & - \frac{\nu^2}{3!.2.64} \Lambda^{kt}\Lambda^{lu}\Lambda^{iv} \left(\stackrel{\curvearrowright}{\oplus}_{kli}\omega_{kr}(R^{\nabla})^r_{lij}  \right)\left(\stackrel{\curvearrowright}{\oplus}_{tuv}\omega\left(R^{\nabla}(X_F,\partial_t)\partial_u,\partial_v\right)\right)y^j. \label{eq:QF5}
\end{eqnarray} 
\end{prop}

\begin{proof}
The terms in $Q(F)$ that will contribute to $*_{\nabla,0}$ up to order three, must be of the form $y^k,y^ky^l,y^ky^ly^p,\nu y^k,\nu y^ky^l$ or $\nu^2 y^k$.

$Q(F)^{(1)}$ and $Q(F)^{(2)}$ are obtained by successive application of $\delta^{-1} \partial$, they don't involve $r$ for degree reason.

For $Q(F)^{(3)}$, the first term in Equation (\ref{eq:QFrecu}) gives $\delta^{-1} \left(\partial Q(F)^{(2)}\right)=\frac{1}{6}\omega\left(\nabla^2_{pk} X_F,\partial_l\right)y^ky^py^l$ and the second one is $\delta^{-1} \left(\frac{1}{\nu}  [r^{(3)}, Q(F)^{(1)}]_{\circ}\right)=-\frac{1}{24}\omega\left(R^{\nabla}(X_F,\partial_p)\partial_k,\partial_l\right)y^ky^py^l$. Using Equation (\ref{eq:Liedernabla}), that describes the Lie derivative of $\nabla$ one obtains the desired expression of $Q(F)^{(3)}$.

Let us analyse the three terms of Equation (\ref{eq:QFrecu}) for $Q(F)^{(4)}$. First, as $\partial$ preserves the degree in $y$, $\delta^{-1}(\partial Q(F)^{(3)})$ is of degree $4$ in $y$ and hence will not contribute. Also, in Equation (\ref{eq:r3}) and (\ref{eq:r4}), we see that $r^{(3)}$, resp. $r^{(4)}$, are of degree $3$, resp. $4$ in $y$. Then, the terms $\delta^{-1}\left(\frac{1}{\nu}[r^{(3)},Q(F)^{(2)}]_{\circ}\right)$ and $\delta^{-1}\left(\frac{1}{\nu}[r^{(4)},Q(F)^{(1)}]_{\circ}\right)$ are both of degree $4$ in $y$ and hence will not contribute.

Inside $Q(F)^{(5)}$, the only terms that will contribute to $*_{\nabla,0}$ up to order $3$ in $\nu$ comes from terms in $\nu^2y^j$. Such terms appear in $\delta^{-1}\left(\frac{1}{\nu}[r^{(3)},Q(F)^{(3)}]_{\circ}\right)$ and $\delta^{-1}\left(\frac{1}{\nu}[r^{(5)},Q(F)^{(1)}]_{\circ}\right)$ and will give Equation (\ref{eq:QF5}). Note that the only term in $r^{(5)}$ that is contributing is the term in $\nu^2y^j$ inside $\delta^{-1}\left(\delta^{-1}(R)\circ\delta^{-1}(R)\right)$.
\end{proof}

\begin{prop}
Modulo terms of order greater or equal than $4$ in $\nu$, the star product $*_{\nabla,0}$ of $F,G \in C^{\infty}(M)$ is given by :
\begin{eqnarray*} \label{eq:*Fleq4}
F*_{\nabla,0}G & = & FG+\frac{\nu}{2}\{F,G\}+\frac{\nu^2}{8} \Lambda^{i_1j_1} \Lambda^{i_2j_2} \nabla^2_{i_1i_2} F\nabla^2_{j_1j_2} G\\
& &+ \frac{\nu^3}{48}\Lambda^{i_1j_1} \Lambda^{i_2j_2} \Lambda^{i_3j_3}\underline{\Lr_{X_F}\nabla}_{i_1i_2i_3}\underline{\Lr_{X_G}\nabla}_{j_1j_2j_3}+O(\nu^4)
\end{eqnarray*}
\end{prop}

\begin{rem}
The general formula for a Fedosov star product up to order $3$ in $\nu$ can be found in  Section 4.5 of \cite{bordemann}.
\end{rem}

\noindent Proposition \ref{prop:trunctrace}, can then be restated in terms of the Fedosov star product $*_{\nabla,0}$.

\begin{theorem} [\cite{fed4},\cite{LLF}]\label{Fedosov density}
Let $\nabla$ be a symplectic connection. A trace density $\rho^{\nabla}$ for the Fedosov star product $*_{\nabla,0}$ is given by :
\begin{equation} \label{eq:tracedensfed}
\rho^{\nabla}:=1+\frac{\nu^2}{24} \mu(\nabla) + O(\nu^3)
\end{equation}

\noindent In particular, if $*_{\nabla,0}$ is closed, then $\nabla$ is a solution to the equation 
$$\mu(\nabla)=C, \textrm{ for } C\in \R.$$
\end{theorem}

\begin{rem} \label{rk:tracefed}
In \cite{fed4}, Fedosov obtain a recursive formula to compute the trace density of a Fedosov star product. As an example of his procedure, he already obtained the Equation (\ref{eq:tracedensfed}), without using the moment map equation for $\mu$.
\end{rem}

\begin{rem} \label{rk:closedandmu}
Although we have seen in Section \ref{subsect:defandgeneral} that star products are closed up to equivalence, it does \emph{not} mean that $\mu$ can always be made constant. Given $*_{\nabla,0}$, it means there exists a closed star product $*$ equivalent to $*_{\nabla,0}$ but $*$ is not necessarily of the form $*_{\nabla',0}$ for some $\nabla'\in \E(M,\omega)$. See Section \ref{sect:CGmomentclosed} for obstructions to the constancy of $\mu$ in the K\"ahler settings.
\end{rem}





\subsection{The Berezin-Toeplitz star product} \label{sect:BTstar}

We work with a closed K\"ahler manifolds $(M,\omega,J)$, with K\"ahler metric $g(\cdot,\cdot):=\omega(\cdot,J\cdot)$, admitting a pre-quantum line bundle $(L,h,\nabla^L)$, that is $L\rightarrow M$ is a holomorphic line bundle with hermitian metric $h$, and Chern connection $\nabla^L$ such that $R^{\nabla^L}=-2\pi i \omega$. Let us recall basic definitions in K\"ahler geometry. For a K\"ahler metric $g = (g_{i{\barj}})$ on a compact K\"ahler manifold $M$ the  Ricci curvature is given by
$$ R_{i{\barj}} = - \frac{\partial^2}{\partial z^i \partial {\overline z}^j} \log \det g.$$
Its trace 
\begin{equation*}\label{scalar curvature}
S = g^{i\barj}R_{i\barj}
\end{equation*}
is called the scalar curvature.

We consider tensor powers $L^k:=L^{\otimes k}$ equipped with induced Hermitian metric $h^k$ and induced connection $\nabla^{L^k}$. The space $\Gamma(M,L^k)$ of smooth sections of $L^k$ is equipped with natural inner product induced by $h^k$. Denote by $L^2(M,L^k)$ the space of $L^2$ sections of $L^k$ and by $H^0(M,L^k)$ its subspace of holomorphic sections. 

The dimension of $H^0(M,L^k)$ is finite. Let $\{s_1,\ldots, s_{N_k}\}$ be a unitary basis of $H^0(M,L^k)$. The Bergman function $\rho:M\rightarrow \R$ is defined by:
\begin{equation} \label{eq:Bergmann}
\rho_k(x):=\sum_{i=1}^{N_k} h^k_x(s_i(x),s_i(x)), \textrm{ for } x\in M.
\end{equation}
When it is constant for all $k>>1$, Cahen--Gutt--Rawnsley \cite{CGR} obtained a deformation quantization of the K\"ahler manifold. 

\noindent The following is a result of Zelditch \cite{Zel} and Lu \cite{Lu}.

\begin{theorem} \label{theor:asymptbergmann}
The Bergman function $\rho_k$ admits an asymptotic expansion as $k\rightarrow +\infty$,
\begin{equation} \label{eq:Bergmannexp}
\rho_k\sim a_0k^m+a_1k^{m-1}+a_2k^{m-2}+\ldots
\end{equation}
where the $a_i$'s are polynomials in the curvature of the K\"ahler manifold and its covariant derivatives. That is, for any $r$ and $s\in \N$, there exists $C_{s,r}>0$ such that:
\begin{equation*}
\left\| \rho_k - \sum_{i=0}^s a_ik^{m-i}\right\|_{C^r}\leq C_{s,r}k^{m-s-1}.
\end{equation*}
In particular, 
\begin{equation*}\label{coefficients}
a_0=1\textrm{ and }  a_1=\frac{1}{4\pi}S
\end{equation*}
where $S$ denotes the scalar curvature.
\end{theorem}

To a function $F \in C^{\infty}(M)$, one can associate a \emph{Toeplitz operator} $T^k_F \in End(H^0(M,L^k))$ defined by
$$ T^k_F: H^0(M,L^k)\rightarrow H^0(M,L^k): s\mapsto \Pi^k(F.s),$$
for $\Pi^k: \Gamma(M,L^k)\rightarrow H^0(M,L^k)$ being the $L^2$-projection. Hereafter are results of Bordemann-Meinrenken-Schlichenmaier, see \cite{Schlich}, that relates Toeplitz operators to star products.

\begin{theorem} \label{theor:BTstarandtrace} \hfill
\begin{enumerate}
\item There exists a unique star product $*_{BT}$ called \emph{Berezin-Toeplitz (BT) star product} defined by
$$F*_{BT}G:=\sum_{j=0}^{\infty}\nu^jC_j(F,G)\textrm{ for } F,G\in C^{\infty}(M)$$
such that 
\begin{equation*}
\left\|   T^k_F\circ T^k_G- \sum_{j=0}^{j=N-1} \left(\frac{1}{k}\right)^jT^k_ {C_j(F,G)} \right\|_{Op}\leq K_N(F,G)\left(\frac{1}{k}\right)^N
\end{equation*}
\item The trace of Toeplitz operators admits an asymptotic expansion of the form:
\begin{equation*}
\left|  Tr^k\left( T^k_F\right) - \sum_{j=0}^{j=N-1}\left(\frac{1}{k}\right)^{j-m} \int_M \tau_j(F)\dvol \right|\leq \tilde{K}_N(F)\left(\frac{1}{k}\right)^{N-m}
\end{equation*}
where the $\tau_j$'s are linear differential operators on $C^{\infty}(M)$, with $\tau_0=Id$. 
\item The trace of the BT star product is given by $\tr^{*_{BT}}(F):= \sum_{j=0}^{\infty}\nu^{j-m} \int_M \tau_j(F)\dvol$.
\end{enumerate}
\end{theorem}

\begin{rem}
The BT star product can also be obtained by Ma--Marinescu's method \cite{MM}.
\end{rem}

\begin{rem} \label{rk:BTfedosov}
The BT star product is known to be of \emph{separation of variables}, that is the $C_j$'s defining it differentiate the first argument in holomorphic direction and the second argument in anti-holomorphic direction. As a consequence of the work of Karabegov parametrising all such star products \cite{Kara96}, $*_{BT}$ can be build using a Fedosov-like construction, Bordemann-Waldmann \cite{BW} and Neumaier \cite{Neu}. 
\end{rem}

Combining point $3$ of the above Theorem and the Tian-Yau-Zelditch expansion, one can see that the ``formalisation'' of the asymptotic expansion of the Bergman function is a trace density for the Berezin-Toeplitz star product. 

\begin{theorem}[Barron-Ma-Marinescu-Pinsonnault \cite{MM2}]
The formal function $\rho(x):=\sum_{r=0}^{+\infty} a_r(x)\nu^{r-m}\in C^{\infty}(M)[\nu^{-1},\nu]]$, where the $a_i$'s come from  expansion \eqref{eq:Bergmannexp}, is a trace density of the Berezin-Toeplitz star product.
\end{theorem}

\begin{proof}
We follow the proof of \cite{MM2}. By definition, 
$$Tr^k\left( T^k_F\right)=\sum_{i=1}^{N_k}\int_M h^k_x(F(x)s_i(x),s_i(x))\dvol=\int_M F(x)\rho_k(x)\dvol.$$
Using the expansion of $\rho_k$, we get
$$Tr^k\left( T^k_F\right) \sim \int_M F(x)\left( a_0k^m+a_1(x)k^{m-1}+a_2(x)k^{m-2}+\ldots\right)\dvol$$
From Theorem \ref{theor:BTstarandtrace}, we know that the coefficients of $\tr^{*_{BT}}(F)$ are given by the asymptotic expansion of $Tr^k\left( T^k_F\right)$.
\end{proof}

\begin{rem}
Related to the Remarks \ref{rk:BTfedosov} and \ref{rk:tracefed}, it should be possible to adapt the Fedosov's computation of trace density \cite{fed4} to obtain a recursive formula for the trace density of the BT star product and hence a recursive formula for the expansion of the Bergman function. It would be interesting to compare it with Lu's method \cite{Lu} to obtain the terms $a_2$ and $a_3$ of the expansion (\ref{eq:Bergmannexp}).
\end{rem}

Visibly, the closedness of the BT star product forces the scalar curvature of the K\"ahler metrics to be constant. Arezzo-Loi-Zuddas \cite{ALZ} proposes to study K\"ahler metrics for which the $a_r$'s are constant and expressed a sufficient condition in terms of balanced metrics. In the language of deformation quantization, Arezzo-Loi-Zuddas studied the closedness of the BT star product. Also, Lu-Tian \cite{LuTian} propose to study K\"ahler metrics for which $a_{m+1}=0$ and obtain a uniqueness result on complex projective spaces.

\begin{rem}
Related to Remark \ref{rk:closedandmu}, the fact that $*_{BT}$ must be equivalent to a closed star product does \emph{not} lead to the existence of constant scalar curvature K\"ahler metrics.
\end{rem}


\section{Yau-Tian-Donaldson conjecture on the existence of constant scalar curvature K\"ahler metrics}


\subsection{Yau-Tian-Donaldson conjecture} \label{sect:YTDconjecture}
Consider a K\"ahler metric $g = (g_{i{\barj}})$ on a compact K\"ahler manifold $M$ with  Ricci curvature $ R_{i{\barj}}$ and scalar curvature $S$ as defined in Section \ref{sect:BTstar}.
If the Ricci curvature 
is proportional to the K\"ahler metric $g$, that is, if there exists a real constant $k$ such that
\begin{equation*}
 R_{i{\barj}} = k g_{i\barj}
\end{equation*}
the metric is called 
a K\"ahler-Einstein metric. 
Obviously, a K\"ahler-Einstein metric is a constant scalar curvature K\"ahler (cscK for short) metric.
The K\"ahler metric 
is called an {\it extremal K\"ahler metric} if  the  $(1,0)$-part 
$$\mathrm{grad}^{1,0}S = \sum_{i,j = 1}^m 
g^{i\barj}\frac{\partial S}{\partial{\overline z}^{j}}
\frac{\partial}{\partial z^i}$$
of the gradient vector field of the scalar curvature $S$ is a holomorphic vector field. 
Obviously, a cscK metric is an extremal K\"ahler metric.
If a K\"ahler-Einstein metric, a cscK metric or  an extrenal K\"ahller metric exists, it is considered as a canonical metric on a compact K\"ahler manifold, and 
it is one of basic problems in K\"ahler geometry to find conditions for the existence of such metrics.

By the Chern-Weil theory the Ricci form 
$$\Ric_{\omega} = \sqrt{-1} \sum_{i,j=1}^{m} R_{i{\barj}} dz^i \wedge d{\bar z}^j $$
represents $2\pi c_1(M)$ for the first Chern class $c_1(M)$ as a de Rham class. 
If a K\"ahler-Einstein metric exists, 
in accordance with the sign of $k$,
$c_1(M)$ is represented by a positive, $0$ or negative
$(1,1)$-form. We express these three cases by writing
$c_1(M) > 0$, $c_1(M) = 0$ or $c_1(M) < 0$.
The condition $c_1(M) > 0$ is also expressed as saying $M$ is a Fano manifold. 
Apparently it is necessary for $M$ to admits a K\"ahler-Einstein metric that one of the three conditions has to be satisfied.
One may ask the converse. 
In \cite{yau78}, Yau proved conversely if $c_1(M) < 0$ then there exists a unique K\"ahler-Einstein metric in the K\"ahler class $-c_1(M)$,
and if $c_1(M) = 0$ then there exists a unique K\"ahler-Einstein metric in each K\"ahler class (Aubin \cite{aubin76} 
also proved the existence in the case of $c_1(M) < 0$ independently). In the case when $c_1(M) > 0$ Chen-Donaldson-Sun \cite{CDS3}
and Tian \cite{Tian12} 
proved that a necessary and sufficient condition for the existence is K-stability. We postpone the definition of K-stability until later in this subsection.
But the general Yau-Tian-Donaldson conjecture for cscK metric is stated as follows. We say a complex line bundle $L \to M$ is ample if $c_1(L) > 0$, and the pair $(M,L)$ 
is called a polarized manifold. Later we will define the notion of K-stability for polarized manifolds. 

In the Fano case we take $L = K_M^{-1}$, and in this case the conjecture was confirmed as mentioned above. For general polarization, 
this conjecture is still unsolved at the moment of this writing, and is being actively studied.
Before the study of the notion of K-stability there were several known necessary conditions and also sufficient conditions. Here we mention two
necessary conditions which are related to our study of deformation quantization.

Let  ${\mathfrak h}(M)$ denote the complex Lie algebra of all holomorphic vector fields on a compact K\"ahler manifold 
$M$. We set
$${\mathfrak h}_{red}(M) = \{ X \in {\mathfrak h}(M)\, |\,X\ \text{has a zero} \}$$
and call it the reduced Lie algebra of holomorphic vector fields. We abbreviate ${\mathfrak h}_{red}(M) $ as ${\mathfrak h}_{red}$ 
sometimes for simplicity.
It is a well-known result (\cite{Lic}, \cite{lebrunsimanca93} or \cite{GauduchonLN}) that for $X \in {\mathfrak h}_{red}(M)$ there exists 
uniquely up to constant functions a complex-valued smooth function 
$u_X$ such that
\begin{equation}\label{eq2}
 i(X) \omega = \sqrt{-1}\,\barpartial u_X.
  \end{equation}
In this sense ${\mathfrak h}_{red}(M)$ coincides with the set of all
``Hamiltonian'' holomorphic vector fields. 
(The terminology ``Hamiltonian'' may be misleading because $X$ does not preserve the symplectic form unless $u_X$ is a pure imaginary valued function).
We always assume that 
Hamiltonian function $u_X$ is normalized as
\begin{equation}\label{eq3}
 \int_M u_X\,\omega^m = 0.
\end{equation}

\begin{theorem}[\cite{calabi85}]\label{decomp}
Let $M$ be a compact extremal K\"ahler manifold. Then the Lie algebra ${\mathfrak h}(M) $ has a semi-direct sum decomposition
$$ {\mathfrak h}(M) = \mathfrak h_0 + \sum_{\lambda > 0}\mathfrak h_{\lambda} $$
where $\mathfrak h_{\lambda}$ is the $\lambda$-eigenspace of $\mathrm{ad}(\sqrt{-1}\mathrm{grad}^{1,0} S)$,
and
$\sqrt{-1}\mathrm{grad}^{1,0}S$ belongs to the center of $\mathfrak h_0$.
Further $\mathfrak h_0$ is reductive, and decomposes as $\mathfrak h_0 = \mathfrak a + \mathfrak h_0 \cap \mathfrak h_{red}$
where $\mathfrak a$ consists of parallel vector fields and thus is abelian.
We also have $\mathfrak h_{red} = \mathfrak h_0 \cap \mathfrak h_{red} +  \sum_{\lambda > 0}\mathfrak h_{\lambda}$. 
\end{theorem}
\noindent
From this theorem it follows that if $M$ admits a constant scalar curvature K\"ahler metric then we have  ${\mathfrak h}(M) = \mathfrak h_0$, and
therefore ${\mathfrak h}(M) $ is reductive. This result is called the Lichnerowicz-Matsushima 
theorem and is a well-known obstruction
for the existence of K\"ahler-Einstein metrics (Matsushima \cite{matsushima57}) and 
K\"ahler metrics of constant scalar curvature (Lichnerowicz \cite{Lic}) in 1950's.

Another obstruction is found by the first author in 1980's \cite{futaki83.1}. Take any K\"ahler class $\Omega := [\omega_0]$ represented by a K\"ahler form $\omega_0$.
Choose any $\omega \in \Omega$. We define a linear map $f : \mathfrak h_{\Omega} \to \bfC$ of the Lie subalgebra $\mathfrak h_{\Omega}$ consisting
of all elements in $\mathfrak h(M)$ preserving $\Omega$ into $\bfC$ by 
\begin{equation}\label{W3}
f(X) :=  \int_M XF \ \omega^m
\end{equation}
where $F \in C^{\infty}(M)$ is given by
$$ \Delta F = S - \frac{\int_M S \omega^m}{\int_M \omega^m}$$
and $XF$ denotes the derivative of $F$ by the holomorphic vector field $X$ and $\Delta=g^{\alpha \bar{\beta}}\partial_{\alpha}\partial_{\bar{\beta}}$.

\begin{theorem}[\cite{futaki83.1}]\label{W4}
Let $M$ be a compact K\"ahler manifold, $\Omega := [\omega_0]$ a fixed K\"ahler class.
Then $f(X)$ given by  (\ref{W3}) 
does not depend on the choice of a K\"ahler form $\omega \in \Omega$.
In particular $f$ is a Lie algebra homomorphism. 
Further, if there exists a constant scalar curvature K\"ahler metric in the K\"ahler class $\Omega$ then we have $f = 0$.
\end{theorem}

There is a Fano manifold satisfying Matsushima's condition of reductiveness but $f \ne 0$, see \cite{futaki83.1}.

Then by (\ref{eq2}), a Hamiltonian holomorphic vector field $X \in \mathfrak h_{red}$ 
is expressed as $X = \mathrm{grad}^{1,0} u_X$. Here, $\mathrm{grad}^{1,0} u_X$ 
is the $(1,0)$-part 
$$\mathrm{grad}^{1,0} u_X = \sum_{i,j = 1}^m 
g^{i\barj}\frac{\partial u_X}{\partial{\overline z}^{j}}
\frac{\partial}{\partial z^i}$$
of the gradient vector field of $u_X$. 
Then the Lie algebra homomorphism  \eqref{W3} is expressed as 
\begin{equation}\label{futaki2}
f(X) = - \int_M u_X S \omega^m. 
\end{equation}

There are many other ways to express the invariant $f(X)$. First we introduce an expression
as an element of equivariant cohomology due to Futaki-Morita \cite{futakimorita85}, see also \cite{futaki88}.

Let $G$ be a complex Lie group, and let $\pi : P \to M$ be a complex analytic
principal $G$-bundle with the right action of $G$.
We assume that the action of the structure group $G$ is a right
action. Each $Y \in \mathfrak g$ defines a complex vector field $Y_\ast$ on $P$ induced by the
action of $G$. 
Suppose that the group $H(M)$ of all automorphisms 
acts on $P$ from the left complex analytically and commuting with the
action of $G$. 

Let $\theta$ be a type (1,0)-connection on $P$ and $\Theta$ its
curvature form. 
Recall that a connection gives by definition a $\mathfrak g$-valued $1$-form
$\theta$ on $P$, called the connection form, such that
\begin{equation*}\label{1.2.3}
\theta(Y_\ast) = Y,\quad \theta(\overline{Y_\ast}) = 0 \qquad \mathrm{for\ every}\ Y \in {\mathfrak g},
\end{equation*}
\begin{equation*}\label{1.2.4}
(R_a)^\ast \theta = \operatorname{ad} (a^{-1}) \theta
\end{equation*}
where $\operatorname{ad}$ denotes the adjoint representation of $G$ on $\mathfrak g$.

Since $H(M)$ acts on $P$, $X \in \mathfrak h(M)$ defines a
holomorphic vector field on $P$, which we shall denote by $X_\ast$.
Since the actions of $G$ and $H(M)$
commutes, $R_{g\ast}X_\ast = X_\ast$. 
Let $I^{k}(G)$ be the set of all $G$-invariant polynomials of degree $k$ on $\mathfrak g$. 
It is shown in \cite{futakimorita85} that, for $\varphi \in I^{m+p}(G)$, $\varphi(\theta(X_{1\ast}), \cdots , \theta( X_{p\ast}), \Theta, \cdots , \Theta)$ is a
well defined $2m$-form on $M$ where $p \ge 0$ and $X_1, \cdots , X_p \in
\mathfrak h(M)$. We define $f_\varphi : \otimes ^p\, \mathfrak h(M) \to \bfC$ by
\begin{equation}\label{5.2.4}
f_\varphi ( X_1 , \cdots , X_p ) = \binom{m+p}{p}\int_M  \varphi(\theta(X_{1\ast}), \cdots , \theta ( X_{p\ast}) , \Theta , \cdots , \Theta )).
\end{equation}
Further, it is shown that 
the definition of $f_\varphi$ is independent of the
choice of the type $(1,0)$ connection $\theta$ and that $f_\varphi$ is an $H(M)$-invariant polynomial 
of degree $p$. Thus 
we get a linear map 
$F : I^{m+p}(G) \to I^p(H(M))$ by defining $F(\varphi) = f_\varphi$.

Let $N$ be a smooth manifold on which a Lie group $G$ acts, and
$E_G \to B_G$ be the universal $G$-bundle. The cohomology group of $N_G = E_G\times_G N$, 
usually denoted by $H^\ast_G (N)$, is called the equivariant cohomology
of $N$. In the special case when $N$ is a point, $H_G^\ast(pt) = H^\ast(BG)$. If
$G$ acts on $N$ freely, $H_G^\ast(N) \cong H^\ast(N/G)$.
Now let $M$ be a compact complex manifold of dimension $m$ and $H$
be the group of all automorphisms of $M$. Let $P \to M$ is a complex
analytic principal bundle whose structure group is a complex Lie group
$G$ with the right action. Assume that the action of $H$ on $M$ lifts
to a left action on $P$ commuting with the action of $G$. Then $P_H = EH \times_H P \to M_H = E H \times_H M$ 
is a principal $G$-bundle.
Then it is shown in \cite{futakimorita85} that 
the following diagram commutes:
$$
\begin{array}{ccc}
I^{m+p}(G) & \mapright{F} & I^p(H)\\
\mapdown{w} & & \mapdown{w}\\
H^{2m+2p}_H(M;\bfC) & \mapright{\pi_\ast} & H^{2p}_H(pt;\bfC)
\end{array}
$$
where the two $w$'s are Weil homomorphisms corresponding to $P_H \to M_H$ and
$E_H \to B_H$, and $\pi_\ast$, denotes the Gysin map (namely integration over the fiber) of $\pi : M_H \to B_H$.

For $p=1$, we may restrict $H$ to $S^1 \subset \bfC^\ast$. Then $B_H = B_{S^1} = \bfC\bfP^\infty$.
Since $H^{2}_H(pt;\bfC) = H^2( \bfC\bfP^\infty;\bfC) = H^2(\bfC\bfP^1;\bfC) = \bfC$, the right hand side of \eqref{5.2.4} can be
considered as an element of $H^{2}_H(pt;\bfC)$ when $X$ is the infinitesimal generator of $S^1$. If we take $\varphi = c_1^{m+1}$
for a Fano manifold $M$ 
where $c_1 = \tr$ we have 
\begin{equation}\label{futaki1}
f_{c_1^{m+1}} (X) = (m+1) \int_M (\mathrm{div}_\eta X/2\pi)\ \Ric_\eta^m
\end{equation}
 where $\eta$ is a K\"ahler form on $M$.
For a K\"ahler form $\omega$ we may take $\omega = \Ric_\eta$ by Calabi-Yau theorem (\cite{yau78}). Then the right
hand side of \eqref{futaki1} becomes $(m+1)$ times \eqref{futaki2}. See page 69, \cite{futaki88} for the proof.

The next expression of the invariant $f(X)$ is due to Donaldson \cite{donaldson02}. Since this expression is used to define
K-stability we formulate it for schemes.
Let $\Lambda \to N$ be an ample line bundle over an $n$-dimensional projective 
scheme $N$. We assume there is a ${\mathbb C}^*$-action as bundle isomorphisms of 
$\Lambda$ covering a ${\mathbb C}^*$-action on $N$. 
For any positive integer $k$, there is an induced $\bfC^*$ action on
$H^0(N, \Lambda^k)$. Put
$d_k = \dim H^0(N, \Lambda^k)$ and let $w_k$ be the weight of $\bfC^*$-action on 
$\wedge^{d_k}H^0(N, \Lambda^k)$. 
For large $k$, 
$d_k$ and $w_k$ are polynomials in $k$ of degree $n$  and $n+1$ respectively
by the Riemann-Roch and the equivariant Riemann-Roch theorems. Therefore 
$w_k/kd_k$ is bounded from above as $k$ tends to infinity.
For sufficiently large $k$ we expand
$$ \frac{w_k}{kd_k} = F_0 + F_1k^{-1} + F_2k^{-2} + \cdots. $$

Now let $L \to M$ is an ample line bundle over a smooth complex manifold $M$ and apply the above formulation
by taking $(\Lambda, N) = (L,M)$ and consider $c_1(L)$ as a K\"ahler class.
We show
\begin{equation}\label{futaki3}
F_1 = \frac{-1}{2m! vol(M, \omega)}f(X)
\end{equation}
when $\sqrt{-1} X$ generates an $S^1$-action.
To show \eqref{futaki3} let us denote by $n$ the complex dimension of $M$.
Expand $d_k$ and $w_k$ as
$$ d_k = a_0k^m + a_1k^{m-1}+ \cdots,$$
$$ w_k = b_0k^{m+1} + b_1k^m + \cdots.$$
Then by the Riemann-Roch and the equivariant Riemann-Roch formulae 
$d_k$ and $w_k$ are computed as degree $0$ and $1$ terms in $t$ of the integral of
\begin{eqnarray*}
e^{k(\omega + t u_X)} Td(\frac{\sqrt{-1}}{2\pi}(tL(\sqrt{-1}X) + \Theta))
= \sum_{p=0}^{\infty} \frac{k^p}{p!}(\omega + t u_X)^p
\sum_{q=0}^{\infty} Td^{(q)}(\frac{\sqrt{-1}}{2\pi}(tL(\sqrt{-1}X) + \Theta))
\end{eqnarray*}
over $M$, c.f. \eqref{5.2.4} or \cite{BGV92}, \cite{BV83}. Here 
$Td^{(q)}$ is the Todd polynomial of degree $q$, $L(X) = \nabla_X - L_X$
and $t$ is the generator of $H^2_{S^1}(pt;\bfZ) = H^2(\bfC\bfP^1;\bfZ)$ of the equivariant cohomology.
Thus we obtain
$$ a_0 = \frac 1{m!}\int_M c_1(L)^m = vol(M), $$
$$ a_1 = \frac 1{2(m-1)!} \int_M \Ric \wedge c_1(L)^{m-1} = \frac 1{2m!} \int_M S\, \omega^m, $$
$$ b_0 = \frac 1{(m+1)!}\int_M (m+1) u_X \omega^m, $$
$$ b_1 = \frac 1{m!} \int_M m u_X \omega^{m-1} \wedge \frac 12 c_1(M) - \frac 1{m!} \int_M 
\operatorname{div}X\ \omega^m .$$
Here, $\omega$ is a K\"ahler form in $c_1(L)$ and $\Ric = \Ric_\omega$ is the Ricci form of $\omega$.
The last term of the previous integral is zero because of the divergence formula. Thus
$$ \frac {w_k}{k d_k} = \frac{b_0}{a_0}(1 + (\frac{b_1}{b_0} - \frac{a_1}{a_0})k^{-1} + \cdots ) $$
from which we have
\begin{eqnarray*}
 F_1 &=& 
 \frac{b_0}{a_0}(\frac{b_1}{b_0} - \frac{a_1}{a_0}) = \frac 1{a_0^2}(a_0b_1 - a_1b_0)\\
 &=& \frac 1{2vol(M)}\int_M u_X(S - \frac 1{vol(M)}\int_M S\, \frac{\omega^m}{m!})\frac{\omega^m}{m!}\\
 &=&  \frac 1{2vol(M)}\int_M u_X \Delta F \frac{\omega^m}{m!}
 \ =\   \frac{ -1}{2vol(M)}\int_M XF \frac{\omega^m}{m!}\\
 &=&  \frac{ -1}{2m!vol(M)}f(X).
\end{eqnarray*}
This competes the proof of \eqref{futaki3}.

Another useful formula is the cohomology formula due to Odaka \cite{odaka13} and Wang \cite{wangxw12}.
Let $X$ be a holomorphic vector field on $M$ which generates an $S^1$-action. 
Suppose the $S^1$-action lifts to a holomorphic action on the total space of an ample line bundle $L \to M$.
Let $\mathcal L \to \bfC\bfP^1$ and $\mathcal M \to \bfC\bfP^1$ be the restriction
to $\bfC\bfP^1 \subset \bfC\bfP^\infty$ of the $L$-bundle $E_{S^1} \times_{S^1} L \to B_{S^1} = \bfC\bfP^\infty$ 
and $M$-bundle $E_{S^1} \times_{S^1} M \to B_{S^1} = \bfC\bfP^\infty$ associated to
the universal $S^1$-bundle. 
Then $f(X)$ can be computed by the intersection number
\begin{equation}\label{intersection no}
f(X) = \frac{m}{m+1} \mu(M, L) c_1(\mathcal L)^{m+1} + c_1(\mathcal L)^m \cdot c_1(K_{\mathcal M/\bfC\bfP^1})
\end{equation}
where
\begin{equation*}\label{slope}
\mu(M,L) := \frac{-c_1(K_M)\cdot c_1(L)^{n-1}}{c_1(L)^n}
\end{equation*}
is the average scalar curvature of a K\"ahler metric in $c_1(L)$. One can show \eqref{intersection no} by 
expressing the equivariant Chern classes as
$$ c_1(\mathcal L) = [\omega + tu_X]\ \ \mathrm{and}\ \  c_1(K_{\mathcal M/\bfC\bfP^1}) = -[\Ric - t \mathrm{div} X].$$
\noindent 
There are other expressions of $f(X)$ such as the degree of CM-line bundle or the degree of Deligne pairing, which
can be shown to coincide by similar computations.

In \cite{tian97} Tian defined the notion of K-stability for Fano manifolds
and proved that if a Fano manifold carries a K\"ahler-Einstein metric then
$M$ is weakly K-stable. Tian's K-stability considers the degenerations of 
$M$ to
normal varieties and uses a generalized version of the invariant $f(X)$.
Note that this generalized invariant is only defined for normal varieties. 
As described above, Donaldson re-defined in \cite{donaldson02} the invariant $f(X)$ 
for projective schemes and also re-defined 
the notion of K-stability for $(M, L)$. The new definition does not require $M$ to be
Fano nor the central fibers of  degenerations to be normal. We now review 
Donaldson's definition of K-stability.
For an ample line bundle $L$ over a projective variety  $M$, a test configuration of
exponent $r$ is a normal polarized variety $(\mathcal M, \mathcal L)$ with the following properties:\\
(1)\ \ there is a $\bfC^*$-action on ${\mathcal M}$ lifting to $\mathcal L$,\\
(2)\ \ there is a flat $\bfC^*$-equivariant morphism $\pi : {\mathcal M} \to \bfP^1$ for the standard $\bfC^*$-action on $\bfP^1$,\\
such that
over $\bfP^1-\{0\}$, $(\mathcal M, \mathcal L)$ is equivariantly isomorphic to $(M \times (\bfC^\ast \cup \{\infty\}), p_M^\ast L^r)$
with the trivial action on the first factor $M$.

The $\bfC^*$-action induces a $\bfC^*$-action on the central fiber
$L_0 \to M_0 = \pi^{-1}(0)$. 
We put $DF(\mathcal M, \mathcal L) := - F_1$ which is called the Donaldson-Futaki invariant of the test configuration
$(\mathcal M, \mathcal L)$. 

If a holomorphic vector field $X$ is the infinitesimal generator of an $S^1$-action on the polarized manifold $(M,L)$, 
the restriction to $\bfP^1$ of $(M_{S^1}, L_{S^1}) = E_{S^1}\times_{S^1} (M,L) \to B_{S^1} = \bfP^\infty$ is a test configuration. This
is called a product test configuration since $M_{S^1}|_{\bfP^1-\{\infty\}} \cong \bfC\times M$ with the diagonal $\bfC^\ast$-action, 
and $DF(M_{S^1},L_{S^1})$ coincides with $f(X)/2vol(M,\omega)$ by \eqref{futaki3}. 


\begin{defi}\ \ $(M,L)$ is said to be K-semistable (resp. stable) if
the $DF(\mathcal M,\mathcal L)$ is non-negative (positive)
for all non-trivial test configurations. $(M,L)$ is said to be
K-polystable if it is K-semistable and $DF(\mathcal M,\mathcal L) = 0$ only
if the test configuration is product. $(M,L)$ is said to be
K-stable if it is K-polystable and the automorphism group of $(M,L)$ is finite.
\end{defi}

\noindent
{\bf Yau-Tian-Donaldson conjecture} : 
For a polarized manifold $(M,L)$, there exists a constant scalar curvature K\"ahler metric in the K\"ahler class $c_1(L)$ if and only if $(M,L)$ is K-polystable.

\begin{rem}
There are other conventions in which K-stable means K-polystable.
\end{rem}
\begin{rem}
Instead of $-F_1$ one may use Odaka-Wang's intersection number in the right hand side of \eqref{intersection no}. See also \cite{LiXu14} and \cite{suzuki16}.
\end{rem}
\begin{rem}
It is known that we may assume $(\mathcal M, \mathcal L)$ is smooth and that the central fiber $M_0$ is reduced, see \cite{DervanRoss}.
\end{rem}
\begin{rem}
K-semistability implies $f(X) = 0$ for any $X$ since both $f(X)$ and $f(-X)$ are non-negative.
\end{rem}
\begin{rem}
Yau-Tian-Donaldson conjecture has been confirmed for Fano manifolds with $L = K_M^{-1}$ (\cite{CDS3}, \cite{Tian12}).
In this Fano case it is known that Donaldson's K-stability is equivalent to Tian's original definition, see \cite{LiXu14}.
\end{rem}

\bigskip


\subsection{Geometric invariant theory and moment map}\label{GITmoment}
The notion of K-stability is modeled on
Geometric Invariant Theory (GIT for short) due to Mumford \cite{mumford}
to construct good moduli space when the equivalence classes are given by orbits of a group action.
The invariant $DF$ is used as the Mumford weight in the Hilbert-Mumford criterion as explained below.
The idea is to discard ``unstable orbits'' and take the quotient of (semi)stable orbits, and
then one will get a Hausdorff and compactifiable moduli space.

There is a moment map interpretation due to Kempf and Ness \cite{KempfNess} (see also \cite{donkro})
of stable orbits.
Let 
$Z$ be a compact K\"ahler manifold with K\"ahler form $\kappa$, and 
$\pi : \Lambda \to Z$ a holomorphic line bundle with $c_1(\Lambda) = [\kappa]$.
Suppose a reductive complex Lie group $G$ is a complexification $K^c$ of a compact
Lie group $K$ where $K$ acts on $Z$ in the Hamiltonian way, i.e. 
for any $X \in \mathfrak k := \mathrm{Lie}(K)$ we have
$$ i(X)\kappa = - d\mu_X$$
for some smooth function $\mu_X \in C^\infty(Z)$.
Then 
$\mu : Z \to \mathfrak k^\ast$ is called the moment map for the action of $K$ if $\mu$ is $K$-equivariant and
$$ \langle \mu, X \rangle = \mu_X. $$
Suppose the action of $K^c$ lifts to $\Lambda$.
Let $p \in Z$.
\begin{defi}
The orbit $K^c\cdot p$ is said to be polystable if, for $\widetilde{p} \in \Lambda^{-1}$ with $\pi(\widetilde{p}) = p$, $\widetilde{p}\ne 0$, 
the orbit $K^c\cdot \widetilde{p}$ in $\Lambda^{-1}$ 
is closed. Note that this is independent of choice of such $\tilde p$. The orbit $K^c\cdot p$ is said to be stable if it is polystable and
$p$ has finite stabilizer.
\end{defi}
\noindent
Kempf-Ness theorem asserts that
the orbit $K^c\cdot p$ is polystable if and only if 
$K^c\cdot p$ has a zero point of $\mu$. That is,
$$ K^c\cdot p \cap \mu^{-1}(0) \ne \emptyset.$$
\noindent
Hilbert-Mumford criterion says that $p \in Z$ is stable with respect to $K^c$-action if and only if 
$p\in Z$ is polystable
with respect to every one parameter subgroup $\sigma : \bfC^\ast \to K^c$.
If $\lim_{t \to 0} \sigma(t)p = p_0$ then $p_0$ is a fixed point of $\sigma$, and $\sigma(t)\Lambda_{p_0} = \Lambda_{p_0}$.
Then $\sigma(t) : \Lambda_{p_0} \to \Lambda_{p_0}$ is a linear action. Let  $\alpha$ be its weight so that
$z \mapsto t^{-\alpha} z$.
Then $p \in Z$ is polystable with respect to $\sigma$ if and only if $\alpha > 0$.
We call the weight $\alpha$ the Mumford weight. Thus $p \in Z$ is polystable if and only if the Mumford weight $\alpha$ is positive for
every one parameter subgroup $\sigma$.

There exists an Hermitian metric $h$ on $\Lambda^{-1}$ such that its Hermitian connection $\theta$ satisfies
$$ - \frac 1{2\pi} d\theta = \pi^{\ast}\kappa.$$
\noindent
We define a function $\ell : K^c\cdot \tilde p \to {\mathbb R}$ on the orbit $K^c\cdot \tilde p \subset \Lambda^{-1}$, $\tilde p \ne 0$, by
\begin{equation}\label{function_ell}
 \ell(\gamma) = \log |\gamma|^2
 \end{equation}
where the norm $|\gamma|$ is taken with respect to $h$. The following is well-known, see \cite{donkro}, section 6.5. 
\begin{itemize}
\item\ \ The function $\ell$ has a critical point if and only if the moment map $\mu : Z \to \mathfrak k^{\ast}$ has a zero on $\Gamma$.
\item\ \ The function $\ell$ is convex.
\end{itemize}
The Donaldson functional in Kobayashi-Hitchin correspondence and Mabuchi K-energy in the study of cscK metrics are modeled on this functional 
$\ell$, and enjoy these two properties.

Suppose we are given a $K$-invariant inner product on ${\mathfrak k}$. Then we have a natural identification
${\mathfrak k} \cong {\mathfrak k}^{\ast}$, and ${\mathfrak k}^{\ast}$ also has a 
$K$-invariant inner product.
Let us consider the function $\phi : K^c\cdot x_0 \to {\mathbb R}$ given by
$\phi(x) = |\mu(x)|^2$. A critical point $x \in K^c\cdot x_0$ of $\phi$ is called an
 {\bf extremal point}.
\begin{prop}[\cite{xwang04}]\label{W2}
Let $x \in K^c\cdot x_0$ be an extremal point. Then we have a decomposition of the Lie algebra
$$ ({\mathfrak k}^c)_x = ({\mathfrak k}_x)^c + \sum_{\lambda > 0} {\mathfrak k}^c_{\lambda} $$
where ${\mathfrak k}^c_{\lambda}$ is the $\lambda$-eigenspace of ${\mathrm ad}(\sqrt{-1}\mu(x))$, and
$\sqrt{-1}\mu(x)$ belongs to the center of $({\mathfrak k}_x)^c$. In particular we have
$({\mathfrak k}_x)^c = ({\mathfrak k}^c)_x$ if $\mu(x) = 0$.
\end{prop}
This is a finite dimensional model of Calabi's decomposition Theorem \ref{decomp}. Calabi's original proof
is given by a Hessian formula of the Calabi functional, the square $L^2$ norm of the scalar curvature, but
L.Wang \cite{Lijing06} gave a finite dimensional model argument of the Hessian formula. This formal argument
is useful since it made us possible to obtain similar decomposition theorem for other geometric nonlinear 
problems which have moment map interpretation, see Theorem \ref{reductiveness} below.

Now we turn to Donaldson-Fujiki picture where $Z$ is an infinite dimensional K\"ahler manifold which we now define.
In the usual study of K\"ahler geometry beginning from Calabi, the complex structure on a compact complex manifold $M$ is fixed, 
some K\"ahler class $[\omega]$ 
of a K\"ahler form $\omega$ is also fixed, and then one tries to find a canonical metric in the K\"ahler class $[\omega]$.
However, in view of Moser's theorem one may fix a symplectic for $\omega$, and 
consider the set of $\omega$-compatible complex structures $J$. The space of such $J$
is our $Z$ in this picture.
Here, we say that $J$ is compatible with $\omega$ if
$$ \omega(JX, JY) = \omega(X,Y), \quad \omega(X,JX) > 0$$
are satisfied for all $X,\ Y \in T_pM$. Therefore, for each $J \in Z$, 
the triple $(M, \omega, J)$ is a K\"ahler manifold. In this situation the tangent space of $Z$ at $J$ is a subspace of the space  $\mathrm{Sym}^2(T^{\ast 0,1}M)$ 
of symmetric tensors of type $(0,2)$, and the natural $L^2$-inner product on $\mathrm{Sym}^2(T^{\ast 0,1}M)$ gives 
$Z$ a K\"ahler structure.

We assume $\dim_{\bfR}M = 2m$. 
The set of all smooth functions $u$ on $M$ with
$$ \int_M u\, \dvol = 0 $$
is a Lie algebra with respect to the Poisson bracket in terms of $\omega$. 
Denote this Lie algebra by ${\mathfrak k}$ and let  $K$  be its Lie group. Namely $K$ is a subgroup of the group
of symplectomorphisms generated by Hamiltonian diffeomorphisms.  $K$ acts on the K\"ahler manifold $Z$ as
holomorphic isometries.
\begin{theorem}[Donaldson-Fujiki]\label{FujDon} Let $S_J$ be the scalar curvature of the K\"ahler manifold $(M, \omega_0, J)$
and let $\mu : Z \to {\mathfrak k}^{\ast}$ be the map given by
$$ <\mu(J), u> = \int_M S_J\, u\, \omega^m $$
where $ u \in \mathfrak k$. Then $\mu$ is a moment map for the action of  $K$ on $Z$.
\end{theorem}
\noindent
Thus, $\mu^{-1}(0)$ is identified with the set of cscK metrics,
and in view of Kempf-Ness theorem, the cscK problem should be a GIT stability issue.
A demerit of this picture is that there is no complexification of the Hamiltonian diffeomorphisms group $K$.
However there is a complexification $\mathfrak k \otimes \bfC$ of $\mathfrak k$. 
If $X$ is a Hamiltonian vector field and $u_X$ its Hamiltonian function, we have
$$L_{JX}\omega = i\partial\barpartial u_X,$$
and thus 
$K^c$-orbit can be considered as the K\"ahler class.

The fact that $\mu$ is a moment map for the action of  $K$ on $Z$ is equivalent to the equation
\begin{equation}\label{mmeq}
\left.\frac{d}{dt}\right|_{t=0}\langle\mu(J_t),u_X\rangle = (JL_XJ, \dot{J})_{L^2} .
\end{equation}
We call this the moment map formula. This formula shows that if $L_XJ = 0$, that is, $X$ is a holomorphic vector field then
the derivative with respect to $J$ of
$$ f(X) =  - \int_M S_J\, u\, \omega_0^m$$
vanishes, and thus $f(X)$ is an invariant independent of $J$, giving an alternative proof of Theorem \ref{W4}.

Note that Theorem \ref{FujDon} or equivalently the equation \eqref{mmeq} also implies that $J$ is a critical point of the Calabi energy
$$ J \mapsto \int_M |S_J|^2 \omega^m $$
if and only if $(M, J, \omega_0)$ is an extremal K\"ahler manifold. This can be seen by taking $u=S_J$.
Using the formal argument of L.Wang, we can give an alternative proof of Calabi's decomposition theorem \ref{decomp}, see \cite{futaki07.1}.


\subsection{Asymptotic Chow semi-stability, balanced embeddings and constant scalar curvature K\"ahler metrics}
Chow stability of  a polarized manifold $(M,L)$ is defined in terms of the stability of the Chow point,
but there is an equivalent description in terms of balanced condition originally due to Luo \cite{Luo}, see also \cite{phongsturm03}.
This balanced condition is already appeared in subsection 2.4. Recall that for a Hermitian metric  $h$ of $L$ 
with its curvature $\omega_h := -\frac{i}{2\pi}\partial\barpartial \log h$ positive,
$s_1,\ \ldots, s_{N_k}$ be an orthonormal basis of $H^0(M,L^k)$ with respect to the $L^2$ inner 
product induced by $h$ and the K\"ahler form $\omega_h$ we defined the Bergman function $\rho_k : M \to \bfR$ by
$\rho_k (x) = \sum_{i=1}^{N_k} || s_i(x)||_{h^k}$. We say that $h^k$ is a balanced metric 
if $\rho_k$ is a constant function. In this case the Kodaira embedding using the orthonormal basis 
 is said to be a balanced embedding. Note that this condition of balanced embedding is equivalent to saying 
 that $(M,\omega) \to (\bfC\bfP^{N_k-1},\omega_{FS})$
is an isometric embedding where $\omega_{FS}$ is the Fubini-Study metric. 
We also say that $(M,L^k)$ is balanced if there is a balanced metric. 
Then $(M,L^k)$ is Chow semistable if and only if
$(M,L^k)$ admits a balanced embedding, Chow stable if it is Chow semistable and the automorphism group of $(M,L^k)$ is finite.
The polarized manifold $(M,L)$ is said to be asymptotically Chow stable (resp. semistable) if for some $\ell$ sufficiently large, 
$(M, L^k)$ is Chow stable (resp. semistable) for all $k \ge \ell$. 

Using the asymptotic expansion Theorem \ref{theor:asymptbergmann} of the Bergman function 
Donaldson \cite{donaldson01} proved the following. 
Let $(M,L)$ be a polarized manifold and suppose that $\mathrm{Aut}(M,L)$ is discrete. If there exists a
constant scalar curvature K\"ahler form in $c_1(L)$ then
\begin{enumerate}
\item $(M,L)$ is asymptotically stable, and thus for each $k$ a balanced metric of $L^k$ 
exists for each $k$, and
\item 
 as $k \to \infty$ the balanced metrics converge to the constant scalar curvature K\"ahler metric. 
\end{enumerate}
This theorem of Donaldson suggests one to try to show the 
existence of a constant scalar curvature K\"ahler metric by using a sequence of balanced metrics.
However the  following result (\cite{Fut}) of the first author shows that when $\mathrm{Aut}(M,L)$ is not discrete it is not always
possible to choose balanced metrics.

Let $I^k(G)$ denote the set of all $G$-invariant polynomials of degree $k$:
$$
I^k(G) = \{ \phi : \mathrm{Sym}^k(\mathfrak g) \to \bfC\ |\ \phi\circ\operatorname{Ad}(g) = \phi \ \text{for any}\  g \in G \}.
$$
We define ${\mathcal F}_{\phi}(X)$ for $\phi \in I^k(G)$ and $X \in \mathfrak h$ by 
\begin{eqnarray}
{\mathcal F}_{\phi}(X) &=& (m-k+1) \int_M \phi(\Theta) \wedge u_X\,\omega^{m-k}
\nonumber
\\ & & + \int_M \phi(\theta(X) + \Theta) \wedge \omega^{m-k+1}.\label{family}
\end{eqnarray}
Then it is shown in \cite{Fut} that ${\mathcal F}_{\phi}(X)$ is independent of the choices of the connection $\theta$ of 
type $(1,0)$ on $P_G$ and of the K\"ahler form $\omega \in \Omega$ on $M$.
In particular $\mathcal F_\phi : \mathfrak h \to \bfC$ is a Lie algebra homomorphism.
If we take $\phi$ to be the $k$-th Todd polynomial $Td^{(k)}$ then 
$\mathcal F_{Td^{(k)}}$, $k=1, \ldots, m$ are obstructions for a polarized manifold $(M,L)$ to asymptotic Chow semistability:
\begin{theorem}[\cite{Fut}]\label{Chow5} If a polarized manifold $(M, L)$ is asymptotically Chow
semistable then for $1 \le \ell \le m$ we have
\begin{equation*}
{\mathcal F}_{Td^{(\ell)}}(X) = 0.
\end{equation*}
In particular, in the case of $\ell=1$ this implies $f(X) = 0$.
\end{theorem}
\noindent
The last statement follows since $Td^{(1)} = c_1/2$ and ${\mathcal F}_{Td^{(1)}}$ coincides with $f_{c_1^{m+1}}$ in \eqref{futaki1} up to 
positive constant.

An example of toric K\"ahler-Einstein manifold satisfying $Td^{(k)} \ne 0$ for some $k$ was suggested by Nill-Paffenholz \cite{NillPaffen}.
The computation of $Td^{(k)} \ne 0$ was undertaken by Ono-Sano-Yotsutani \cite{OSY09}. 
In fact, it turns out that ${\mathcal F}_{Td^{(1)}}(X) = 0$ (because it is K\"ahler-Einstein)
but for $\ell \ge 2$ we have ${\mathcal F}_{Td^{(\ell)}}(X) \ne 0$.

Further, Della Vedova-Zuddas \cite{DVZ10} gave an example 
of a compact K\"ahler surface with constant scalar curvature K\"ahler form belonging to
an integral class which is asymptotically Chow unstable.

More recently Sano and Tipler \cite{SanoTipler17}
showed if there exists an extremal K\"ahler metric in $c_1(L)$ of a polarized
manifold $(M,L)$, there is a $\sigma$-balanced metric for some $\sigma \in \mathrm{Aut}(M,L)$ for each $k$,
and the sequence of $\sigma$-balanced metric converges to the extremal K\"ahler metric
where $\sigma$-balanced metric is defined by
$$ \Phi_k^\ast \omega_{FS} = \sigma^\ast \omega_h,$$
$\Phi_k$ is the Kodaira embedding by $L^2(h)$ basis as before.


\section{Cahen-Gutt moment map and closed Fedosov star product} \label{sect:CGmomentclosed}
Let $(M,\omega)$ be a compact symplectic manifold. 
In subsection \ref{sect:momenttrace} we defined $\mathcal E(M,\omega)$ to be the space of all symplectic connections,
the symplectic form $\Omega^{\mathcal E}$ on $\mathcal E(M,\omega)$, and the Cahen-Gutt moment map 
$\mu : \mathcal E(M,\omega) \to C^\infty(M)$ with respect to the Hamiltonian group action, see Theorem \ref{theor:momentE}.

Now we assume that $M$ is a compact K\"ahler manifold and that $\omega$ is a fixed symplectic form. We have set $Z$ in 
subsection \ref{GITmoment} to be 
$$ Z = \{ J\ \text{integrable complex\ structure}\ |\ (M,\omega, J)\ \text{is\ a\ K\"ahler\ manifold}\}.$$
The second author considered in \cite{LLF}, \cite{La Fuente-Gravy 2016_2} the {\it Levi-Civita map}
$lv : Z \to \mathcal E(M,\omega)$ sending $J$ to the Levi-Civita connection $\nabla^J$
of the K\"ahler manifold $(M,\omega,J)$. The pull-back of the Cahen-Gutt moment map is given by:
\begin{equation*}
(lv^*\mu)(J)=2\Delta^J S_J + P(\nabla^J),
\end{equation*}
for $S_J$ being the scalar curvature of the K\"ahler manifold $(M,\omega,J)$ and $\Delta^J:=(g^J)^{\alpha \bar{\beta}}\partial_{\alpha}\partial_{\bar{\beta}}$ where $g^J(\cdot,\cdot):=\omega(\cdot,J\cdot)$.

In \eqref{eq:Liedernabla}, $\Lr_{X_f}\nabla$ is expressed as 
\begin{equation*}\label{Cahen-Gutt2}
\underline{\Lr_{X_f}\nabla} = (\omega_{uv}R^v_{tsq}X_f^s + \nabla_q\nabla_u X_f^s\,\omega_{st})\,dx^q \otimes dx^u \otimes dx^t
\end{equation*}
in real coordinates, 
if we choose local holomorphic coordinates $z^1, \cdots, z^m$ then
it is expressed as 
\begin{eqnarray}\label{infinitesimal}
\underline{\Lr_{X_f}\nabla^J} &=& 
 f_{ijk} dz^i \otimes dz^j \otimes dz^k + f_{\bari\barj\bark} dz^\bari \otimes dz^\barj \otimes dz^\bark \\
&& + f_{ij\bark} dz^i \otimes dz^j \otimes dz^\bark + f_{\bari\barj k} dz^\bari \otimes dz^\barj \otimes dz^k \nonumber\\
&& + f_{ik\barj} dz^i \otimes dz^\barj \otimes dz^k + f_{\bari\bark j} dz^\bari \otimes dz^j \otimes dz^\bark \nonumber\\
&& + f_{jk\bari} dz^\bari \otimes dz^j \otimes dz^k + f_{\barj\bark i} dz^i \otimes dz^\barj \otimes dz^\bark \nonumber
\end{eqnarray}
where the lower indices of $f$ stand for the covariant derivatives, e.g. $f_{ij\bark} = \nabla_\bark \nabla_j \nabla_i f$,
see \cite{FO_CahenGutt}.

Since the terms in the right hand side of \eqref{infinitesimal} are pointwise linearly independent, we obtain the following.
\begin{prop}\label{nondegenerate}For a real smooth function $f$, 
$L_{X_f}\nabla^J = 0$ if and only if $L_{X_f}J = 0$. In this case, $X_f$ is a holomorphic Killing vector field.
\end{prop}
Hence from the moment map formula \eqref{eq:momentmu}, the above Proposition \ref{nondegenerate} and Theorem \ref{Fedosov density},
we obtain the following Theorem.
We consider $\mathfrak h_{\mathbf R}$ consisting of vector fields $X$ such that $\mathrm{grad}^{(1,0)}f \in \h_{red}$ for some real smooth function, normalised by $\int_Mf\ \omega^m=0$.
\begin{theorem}[\cite{La Fuente-Gravy 2016_2}]\label{Futaki}
Let $(M, \omega)$ be a compact K\"ahler manifold, and $\mathfrak h_{\mathbf R}$ be the real reduced Lie algebra of holomorphic vector fields. 
Then 
$$ \mathrm{Fut}(\mathrm{grad}^{(1,0)}f) := \int_M \mu(\nabla^J)\, f\ \omega^m $$
is independent of the choice of $J \in \mathcal J(M,\omega)$. If $\mathrm{Fut} \ne 0$ then there is no K\"ahler metric 
for which the Fedosov star product $\ast_{\nabla,0}$ is closed.
\end{theorem}
\noindent
In \cite{La Fuente-Gravy 2016_2}, this theorem was proven by $J$-fixed and $\omega$-varying argument.
It is also shown in \cite{La Fuente-Gravy 2016_2} that the character $\mathrm{Fut}$ coincides with the imaginary part
of $\mathcal F_{\frac{8\pi^2}{(m-1)!}(c_2 - \frac12 c_1^2)}$ in \eqref{family}. Note also that $c_1^2-2c_2 $ is the first Pontrjagin class.

\begin{theorem}[\cite{FO_CahenGutt}]\label{reductiveness}
 Let $M$ be a compact K\"ahler manifold. If there exists a K\"ahler metric with non-negative Ricci curvature such that $\mu(\nabla)$ is constant for the Cahen--Gutt moment map $\mu$ and the Levi-Civita connection $\nabla$ then
 the reduced Lie algebra $\h_{red}$ of holomorphic vector fields is reductive. In particular, if $\h_{red}$
 is not reductive then there is no K\"ahler metric with non-negative Ricci curvature such that the Fedosov star product $\ast_{\nabla,0}$ for the Levi-Civita
 connection $\nabla$ is closed.
\end{theorem}
To show this we define Cahen--Gutt version of extremal K\"ahler metrics and prove a similar structure theorem as 
the Calabi extremal K\"ahler metrics.
 The strategy of the proof of the structure theorem for Cahen--Gutt extremal K\"ahler manifold is to use the formal 
finite dimensional argument for the Hessian formula of the squared norm of the moment map
given by Wang \cite{Lijing06}. The merit of Wang's argument is that once the suitable modification of the 
Lichnerowicz operator is made we can apply his formal argument without using the explicit expression of
the modified Lichnerowicz operator. One only has to identify the kernel of the Lichnerowicz operator, up to now we can only do that when the Ricci curvature is non-negative. This strategy has been used previously for perturbed extremal K\"ahler metrics
in \cite{futaki07.1}. 


Denoting by $H_{red}$ the connected Lie group whose Lie algebra is $\h_{red}$.

\begin{theorem} \label{theor:dimension}
Let $(M,\omega,J)$ be a closed K\"ahler manifold with non-negative Ricci curvature such that $\mu(\nabla)$ is constant.
Then, there exists a neighbourhood $U$ of $\omega \in \MOm$ such that if $\omega'\in U$ with $\mu'(\nabla')$ is constant then $\omega'$ lies in the $H_{red}$ orbit of $\omega$.
In particular, if $*_{\nabla,0}$ is closed, the only closed Fedosov star product of the form $*_{\nabla',0}$ for $\omega'\in U$ are isomorphic to $*_{\nabla,0}$. The isomorphism is the pull back by an element of $H_{red}$.
\end{theorem}
\noindent
Note that when we wrote $\mu'(\nabla')$ we mean the Cahen-Gutt moment map $\mu'$ on $\E(M,\omega')$. Similarly for the Fedosov star product, $*_{\nabla',0}$ is a deformation quantization of the symplectic manifol $(M,\omega')$.


Attempting to formulate K-stability for the existence problem of K\"ahler metrics with constant 
Cahen-Gutt momentum $\mu(\nabla)$, we need to know the correct sign convention for
the Donaldson-Futaki invariant. This sign convention is determined by the sign convention of the K-energy
so that it is convex on the space of K\"ahler forms. Note that the K-energy plays the role of the function
$\ell$ in \eqref{function_ell} so that it must be convex. The correct sign convention of the K-energy is
checked as follows.

The smooth path $\omega_{\phi_t}$ is a geodesic if and only if 
$\ddot{\phi} - \|d\phi\|^2_t=0$. We propose a $K$-energy with respect to the Cahen-Gutt moment map $\mu$ by
$$\mathcal{K}(\omega_{\varphi},\omega_0):=\int_{0}^{1} \int_M \dot{\varphi}(\mu(\nabla^t)-\mu_0)\dvol dt,$$
integration is taken along a path joining $\omega_0$ to $\omega_{\varphi}$.
(Note that this convention is opposite to the cscK case where $\mu(\nabla)$ is the scalar curvature.)
Its differential is given by 
$$d\mathcal{K}(\dot{\varphi})=\int_M \dot{\varphi}\mu(\nabla^t)\dvol.$$
The Hessian of $\mathcal{K}$ along a geodesic $\omega_{\phi_t}$ is then computed using the geodesic equation here above and the moment map equation. Take $f_t$ the family of diffeomorphisms generated by the vector field $-\grad^{\phi}(\dot{\phi})$ and $J_t:=f^{-1}_*\circ J \circ f_{t*}$. We have
$$\textrm{Hess}(\mathcal{K})(\dot{\phi},\dot{\phi}) = -(\LC^*\Omega)\left( \Lr_{X_{f_t^*\dot{\phi}}}J_t, J_t\Lr_{X_{f_t^*\dot{\phi}}}J_t  \right),$$
which is non-negative provided the Ricci tensor of $\omega_{\phi_t}$ is non-negative \cite{LLF}.

For $Y$ generating a Hamiltonian isometric $S^1$-action with 
\begin{equation}\label{eq:newconv}
i(Y)\omega=-d v_Y.
\end{equation}
Then $\mathrm{grad}^{(1,0)}v_Y\in \h$, the character $\mathrm{Fut}$ admits the following expression in terms of equivariant cohomology classes:
\begin{eqnarray}\label{Fut_1}
\mathrm{Fut}(\mathrm{grad}^{(1,0)}v_Y) & = & 2m \int_M (v_Y+\omega)^{m-1}c_2\left(\sqrt{-1}\left(R+\nabla Y^{(1,0)}\right)\right) \nonumber\\
& & -  m \int_M(v_Y+\omega)^{m-1}(Ric_{\omega}-\Delta v_Y)^2.
\end{eqnarray}
Adapting the cohomology formula \eqref{intersection no} of Odaka \cite{odaka13} and Wang \cite{wangxw12} to our context, we obtain
\begin{equation}\label{Fut_2}
\frac{1}{(2\pi)^m}\mathrm{Fut}(\mathrm{grad}^{(1,0)}v_Y)= \frac{-2}{m+1}\kappa(M,L).c_1(\mathcal{L})^{m+1}+2m\left( c_2(\mathcal{M})- \frac{1}{2}c_1^2(\mathcal{K}^{-1}_{\mathcal{M}/\bfC\bfP^1})\right).c_1(\mathcal{L})^{m-1}
\end{equation}
where $\kappa(M,L)$ is the average of the Cahen-Gutt momentum
\begin{equation*}
\kappa(M,L):= m(m-1)\frac{\left(c_2-\frac{1}{2}c_1^2\right)(M).c_1(L)^{m-2}}{c_1(L)^m}.
\end{equation*}
Note that the first term in \eqref{Fut_2} did not appear in \eqref{Fut_1} because we assumed the normalization \eqref{eq3} for $v_Y$.

This equivariant cohomology formula suggests that one could define $K$-stability related to the study of K\"ahler metric with constant Cahen-Gutt momentum, at least if one can restricts to smooth test configurations as in \cite{DervanRoss}.

\begin{rem}
The two conventions \eqref{eq2} and \eqref{eq:newconv} agree when $Y=JX$ and $u_X=v_Y$.
\end{rem}


\begin{thebibliography}{99}

\bibitem{ALZ} C. Arezzo, A. Loi, F. Zuddas : On homothetic balanced metrics, Ann. Glob. Anal. Geom. {\bf 41}, 473--491 (2012).

\vspace{-0.2cm}

\bibitem{aubin76}T.~Aubin : Equations du type de Monge-Amp\`ere sur les 
vari\'et\'es k\"ahl\'eriennes compactes, C. R. Acad. Sci. Paris, {\bf 283}, 
119--121 (1976).




\vspace{-0.2cm}

\bibitem{MM2}  T. Barron, X. Ma, G. Marinescu, M. Pinsonnault : Semi-classical properties of Berezin-Toeplitz operators with $C^k$-symbol, J. Math. Phys. {\bf  55} 042108 (2014).




\vspace{-0.2cm}

\bibitem{BFFLS} F. Bayen, M. Flato, C. Fronsdal, A. Lichn\'erowicz, D. Sternheimer : Deformation theory
and quantization, Annals of Physics {\bf 111}, part I : 61--110, part II : 111--151 (1978).  

\vspace{-0.2cm}

\bibitem{BGV92}
N.~Berline, E.~Getzler, M.~Vergne : Heat kernels and Dirac operators. 
Grundlehren der Mathematischen Wissenschaften, 298.
Springer-Verlag, Berlin, 1992. 

\vspace{-0.2cm}

\bibitem{BV83}N.~Berline, M.~Vergne : Z\'eros d'un champ de
vecteurs et classes characteristiques equivariantes, Duke
Math. J. {\bf 50}, 539--549 (1983).

\vspace{-0.2cm}

\bibitem{bertcagutt} M. Bertelson, M. Cahen, S. Gutt : Equivalence of star products, Class. Quan. Grav. {\bf 14}, A93--A107 (1997).

\vspace{-0.2cm}

\bibitem{bordemann} M. Bordemann : (Bi)Modules, morphisms, and reduction of star-products: the symplectic case, foliations, and obstructions, Trav. Math. {\bf 16}, 9--40 (2005).

\vspace{-0.2cm}

\bibitem{BMS} M. Bordemann, E. Meinrenken, M. Schlichenmaier : Toeplitz quantization of K\"ahler
manifolds and $\mathfrak{gl}_n$, $n\rightarrow +\infty$ limits, Comm. Math. Phys. {\bf 165}, 281--296  (1994).

\vspace{-0.2cm}

\bibitem{BRW} M. Bordemann, H. R\"omer, S. Waldmann : A remark on formal KMS states in deformation quantization, Lett. Math. Phys. {\bf 45}, 49--61 (1998) .

 
\bibitem{BW} M. Bordemann, S. Waldmann : A Fedosov star product of the Wick type for K\"ahler manifolds, Letters in Mathematical Physics {\bf 41} (3), 243--253 (1997).


\vspace{-0.2cm}




 


\bibitem{cagutt} M. Cahen, S. Gutt : Moment map for the space of symplectic connections, Liber Amicorum Delanghe, F. Brackx and H. De Schepper eds., Gent Academia Press, 2005, 27--36.

\vspace{-0.2cm}

\bibitem{CGR} M. Cahen, S. Gutt, J.Rawnsley : Quantization of K\"ahler manifolds, Part I: Journal of Geometry and Physics {\bf 7}, 45--62 (1990), Part II: Transactions A.M.S. {\bf 337}, 73--98  (1993), Part III: Lett. in Math. Phys. {\bf 30}, 291--305 (1994), Part IV: Lett. in Math. Phys. {\bf 34}, 159--168 (1995).



\vspace{-0.2cm}

\bibitem{calabi85} E. Calabi : Extremal K\"ahler metrics II, Differerential Geometry and Complex Analysis, Springer-Verlag, 95--114 (1985).

\vspace{-0.2cm}

\bibitem{CDS3}X.~X.~Chen, S.~K.~Donaldson, S.~Sun : 
K\"ahler-Einstein metric on Fano manifolds. III: limits with cone angle approaches $2\pi$ and completion of the main proof, 
J. Amer. Math. Soc. {\bf 28}, 235--278 (2015). 

\vspace{-0.2cm}

\bibitem{CFS} A. Connes, M. Flato, D. Sternheimer : Closed star products and cyclic cohomology, Lett. in Math. Phys {\bf 24}, 1--12 (1992).


\vspace{-0.2cm}


\bibitem{del} P. Deligne : D\'eformation de l'alg\`ebre des fonctions d'une vari\'et\'e symplectique : comparaison entre Fedosov et De Wilde, Lecomte., Selecta Math. {\bf 1}, 667--697 (1995).

\vspace{-0.2cm}

\bibitem{DVZ10}A. Della Vedova and F. Zuddas : Scalar curvature and asymptotic Chow stability of projective bundles and blowups. Trans. Amer. Math. Soc. {\bf 364}, no. 12, 6495--6511 (2012). 


\vspace{-0.2cm}

\bibitem{DervanRoss}
R.~Dervan, J.~Ross : K-stability for K\"ahler manifolds, 
Math. Res. Lett., 
{\bf 24}, No.3, 689--739 (2017). 


\vspace{-0.2cm}

\bibitem{DWL} M. De Wilde, P.B.A. Lecomte : Existence of star-products and of formal deformations of the Poisson Lie Algebra of arbitrary symplectic manifolds., Lett. Math. Phys. {\bf 7}, 487--496 (1983).

\vspace{-0.2cm}

\bibitem{donaldson01}S.K.~Donaldson : Scalar curvature and projective
embeddings, I, J. Differential Geometry {\bf 59}, 479--522 (2001).

\vspace{-0.2cm}

\bibitem{donaldson02}S.K.~Donaldson : Scalar curvature and stability of toric
varieties, J. Differential Geometry {\bf 62}, 289--349 (2002).

\vspace{-0.2cm}

\bibitem{donkro}S.K.~Donaldson, P.B.~Kronheimer : The geometry of four manifolds, Oxford Mathematical Monographs, Claren Press, Oxford, 1990.

\vspace{-0.2cm}




\vspace{-0.2cm}

\bibitem{fed2} B.V. Fedosov : A simple geometrical construction of deformation quantization, Journal of Differential Geometry {\bf 40}, 213-238 (1994).

\vspace{-0.2cm}

\bibitem{fed} B.V. Fedosov : {\it Deformation quantization and index theory}, Mathematical Topics vol. 9, Akademie Verlag, Berlin, 1996.

\vspace{-0.2cm}

\bibitem{fed3} B.V. Fedosov : Quantization and The Index, Dokl. Akad. Nauk. SSSR {\bf 291}, 82--86 (1986).

\vspace{-0.2cm}

\bibitem{fed4} B.V. Fedosov : On the trace density in deformation quantization, in Deformation quantization (Strasbourg, 2001), vol. 1 of IRMA Lect. Math. Theor. Phys., de Gruyter, Berlin, 2002,  67--83.

\vspace{-0.2cm}


 \bibitem{futaki83.1}A.~Futaki : 
An obstruction to the existence of Einstein K\"ahler metrics, Invent. 
Math. {\bf 73}, 437--443 (1983).


\vspace{-0.2cm}

\bibitem{futaki83.2}A.~Futaki : On compact K\"ahler manifolds of constant scalar curvature, 
Proc. Japan Acad., Ser. A, {\bf 59}, 401--402 (1983).


\vspace{-0.2cm}

\bibitem{futaki88}A.~Futaki : K\"ahler-Einstein metrics and integral invariants,
Lecture Notes in Math., vol.1314, Springer-Verlag, Berline-Heidelberg-New York,(1988).


\vspace{-0.2cm}

\bibitem{Fut} A. Futaki : Asymptotic Chow semi-stability and integral invariants, 
Internat. Journ. of Math., {\bf 15}(9), 967--979 (2004).


\vspace{-0.2cm}

\bibitem{futaki07.1}A.~Futaki : Holomorphic vector fields and perturbed extremal K\"ahler metrics, 
J. Symplectic Geom. {\bf 6}, No. 2, 127-138 (2008).

\vspace{-0.2cm}

\bibitem{futakimorita85}A.~Futaki, S.~Morita : Invariant polynomials of the automorphism 
group of a compact complex manifold, J. Differential. Geom. {\bf 21}, 135--142 (1985).

\vspace{-0.2cm}

\bibitem{FO_CahenGutt}A.~Futaki, H.~Ono : 
Cahen-Gutt moment map, closed Fedosov star product and structure of the automorphism group, 
to appear in J. Symplectic Geom. arXiv1802.10292. 








\vspace{-0.2cm}


 \bibitem{GauduchonLN}P. Gauduchon : Calabi's extremal metrics: An elementary introduction, Lecture Notes.


\vspace{-0.2cm}

\bibitem{gr3} S. Gutt, J. Rawnsley : Natural star products on symplectic manifolds and quantum moment maps, Lett. in Math. Phys. {\bf 66}, 123--139 (2003).

\vspace{-0.2cm}

\bibitem{gr} S. Gutt, J. Rawnsley : Traces for star products on symplectic manifolds, Journ. of Geom. and Phys. {\bf 42}, 12--18 (2002).

\vspace{-0.2cm}

\bibitem{Kara96} A.V. Karabegov : Deformation Quantizations With Separation of Variables on a Kahler Manifold, Comm. in Math. Phys. {\bf 180}, 745--755 (1996).

\vspace{-0.2cm}

\bibitem{Kara} A.V. Karabegov : On the Canonical Normalization of a Trace Density of Deformation Quantization, Lett. in Math. Phys. {\bf 45}, 217--228 (1998).

\vspace{-0.2cm}

\bibitem{karaschlich} A.V. Karabegov, M. Schlichenmaier : Almost K\"ahler Deformation Quantization, Lett. Math. Phys. {\bf 57}, 135--148 (2001).


\vspace{-0.2cm}

\bibitem{KempfNess}G.~Kempf and L.Ness : On the lengths of vectors in representation spaces, Lecture Notes in Math., {\bf 732}, pp. 233--242, Springer-Verlag.

\vspace{-0.2cm}

\bibitem{Kon} M. Kontsevitch : Deformation Quantization of Poisson Manifolds. I. Preprint q-alg/9709040, September 1997.


\vspace{-0.2cm}


\bibitem{LLF} L. La Fuente-Gravy : Infinite dimensional moment map geometry and closed Fedosov's star products, Ann. of Glob. Anal. and Geom. {\bf 49} (1), 1--22 (2015).

\vspace{-0.2cm}

\bibitem{La Fuente-Gravy 2016_2}
L. La Fuente-Gravy : Futaki invariant for Fedosov's star products. to appear in Journ. of Sympl. Geom., arXiv:1612.02946.

\vspace{-0.2cm}

\bibitem{lebrunsimanca93}C.~LeBrun, R.S.~Simanca : Extremal K\"ahler metrics and
complex deformation theory, Geom. Func. Analysis {\bf  4}, 298--336 (1994).

\vspace{-0.2cm}

\bibitem{LiXu14}
C.~Li, C.~Xu : Special test configurations and K-stability of Q-Fano varieties, Ann. of
Math. {\bf 180}, no.1, 197--232 (2014).

\vspace{-0.2cm}

\bibitem{Lic}A.~Lichnerowicz : G\'eometrie des groupes de transformations,
Dunod, Paris (1958).

\vspace{-0.2cm}

\bibitem{Lu} Z. Lu : On the lower order terms of the asymptotic expansion of Tian-Yau-Zelditch, Amer. J. Math. {\bf 122} (2), 235--273 (2000).

\vspace{-0.2cm}

\bibitem{LuTian} Z. Lu, G. Tian : The log term of the Szeg\"o Kernel, Duke Math. J.
{\bf 125} (2), 351--387 (2004).

\vspace{-0.2cm}

\bibitem{Luo}
H.~Luo : Geometric criterion for Gieseker-Mumford stability of polarized manifolds, J. Differential Geom. {\bf 49}, no. 3, 577--599 (1998).

\vspace{-0.2cm}

\bibitem{MM} X. Ma, G. Marinescu : Berezin-Toeplitz quantization on K\"ahler manifolds, Journal f\"ur die reine und angewandte Mathematik (Crelles Journal), {\bf 2012} (662), Pages 1--56 (2012).

\vspace{-0.2cm}

\bibitem{matsushima57} Y.~Matsushima : Sur la structure du groupe 
d'hom\'eomorphismes analytiques d'une certaine vari\'et\'e kaehl\'erienne, Nagoya
Math. J. {\bf 11}, 145--150 (1957).

\vspace{-0.2cm}

\bibitem{mumford}D.~Mumford : Stability of projective varieties, L'Enseignement
Mathematiques, {\bf 23}, 39--110 (1977).

\vspace{-0.2cm}

\bibitem{NT} R. Nest, B. Tsygan : Algebraic index theorem for families, Advances in Math. {\bf 113}, 151--205  (1995).

\vspace{-0.2cm}

\bibitem{Neu} N. Neumaier : Universality of Fedosov's Construction for Star Products of Wick Type on Pseudo-K\"ahler Manilfolds, Reports on Mathematical Physics {\bf 52}, (1), 43--80 (2003).

\vspace{-0.2cm}

\bibitem{NillPaffen}B.~Nill, A.~Paffenholz : Examples of K\"ahler-Einstein toric Fano manifolds associated to non-symmetric reflexive polytopes, Beitr. Algebra Geom. {\bf 52}, no. 2, 297--304 (2011).

\vspace{-0.2cm}

\bibitem{odaka13}
Y.~Odaka : A generalization of the Ross-Thomas slope theory, Osaka J. Math. {\bf 50},
no. 1, 171--185 (2013).



\vspace{-0.2cm}

\bibitem{OMY} H. Omori, Y. Maeda, A. Yoshioka : Weyl manifolds and deformation quantization, Adv. in Math. {\bf 85}, 224--255, 1991.

\vspace{-0.2cm}

\bibitem{OMY2} H. Omori, Y. Maeda, A. Yoshioka : Existence of a closed star product, Lett. Math. Phys {\bf 26}, 285--294, 1992.

\vspace{-0.2cm}

\bibitem{OSY09} H.Ono, Y.Sano, N.Yotsutani : An example of asymptotically Chow unstable manifolds with constant scalar curvature, 
Annales de L'Institut Fourier {\bf 62}, no.4, 1265--1287 (2012).

\vspace{-0.2cm}

\bibitem{phongsturm03}D.H.~Phong, J.~Sturm : Stability, energy functionals,
and K\"ahler-Einstein metrics, Comm. Anal. Geom. {\bf 11}, 563--597 (2003).



\vspace{-0.2cm}

\bibitem{Rawnsley} J. Rawnsley : Coherent states and K\"ahler manifolds, Quart. J. Math., Oxford Ser.(2) {\bf 28}, 403--415 (1977).

\vspace{-0.2cm}

\bibitem{SanoTipler17}
Y.~Sano and C.~Tipler : A moment map picture of relative balanced metrics on extremal K\"ahler manifolds,
 preprint. arXiv 1703.09458.

\vspace{-0.2cm}

\bibitem{Schlich} M. Schlichenmaier : Berezin-Toeplitz quantization of compact K\"ahler manifolds, in Quantization, Coherent States and Poisson Structures, Proceedings of the 14th Workshop on Geometric Methods in Physics (Bialowieza, Poland, July 1995).


\vspace{-0.2cm}

\bibitem{suzuki16}
Y.~Suzuki : Cohomology formula for obstructions to asymptotic Chow semistability, Kodai Math. J. {\bf 39}, no. 2, 340--353 (2016). 

\vspace{-0.2cm}

\bibitem{tian97}G.~Tian : K\"ahler-Einstein metrics with positive scalar
curvature, Invent. Math. {\bf 130}, 1--37 (1997).

\vspace{-0.2cm}

\bibitem{Tian} G. Tian : On a set of polarized K\"ahler metrics on algebraic
manifolds, J. Differential Geom. {\bf 32} (1), 99--130 (1990).

\vspace{-0.2cm}

\bibitem{Tian12}G.~Tian: K-Stability and K\"ahler-Einstein Metrics. Comm. Pure Appl. Math. {\bf 68}, no. 7: 1085--1156 (2015).





\vspace{-0.2cm}

\bibitem{Lijing06} L.-J. ~Wang : 
Hessians of the Calabi functional and the norm function,
Ann. Global Anal. Geom. {\bf 29}, No.2, 187--196 (2006).

\vspace{-0.2cm}

\bibitem{xwang04}X.-W.~Wang : 
Moment maps, Futaki invariant and stability of projective manifolds, 
Comm. Anal. Geom. {\bf 12}, no. 5, 1009--1037 (2004).

\vspace{-0.2cm}

\bibitem{wangxw12}
X.-W.~Wang : Height and GIT weight, Math. Res. Lett. {\bf 19}, no. 4,
909--926 (2012).

\vspace{-0.2cm}

\bibitem{yau78}S.-T.Yau : On the Ricci curvature of a compact K\"ahler
manifold and the complex Monge-Amp\`ere equation I, Comm. Pure Appl.
Math. {\bf 31}, 339--441 (1978).


\vspace{-0.2cm}

\bibitem{Zel} S. Zelditch : Asymptotics of holomorphic sections of powers
of a positive line bundle. In {\it S\'eminaire sur les \'Equations aux
D\'eriv\'ees Partielles}, 1997--1998, pages Exp. No. XXII, 12. \'Ecole
Polytech., Palaiseau (1998).












































 












\end{thebibliography}
\end{document}